%% file: RKLMC_revisited_neurips2023.tex
\documentclass[a4paper]{article}

\usepackage[nonatbib,preprint]{neurips_2023}

\usepackage[utf8]{inputenc} 
\usepackage[T1]{fontenc}    
\usepackage{hyperref}       
\usepackage{url}            
\usepackage{booktabs}       
\usepackage{amsfonts}       
\usepackage{nicefrac,amssymb,amsmath}       
\usepackage{microtype}      
\usepackage{cleveref}    
\usepackage{autonum}  
\usepackage{xcolor,colortbl}
\usepackage{amsthm}
\usepackage{comment}
\usepackage[none]{hyphenat}

\definecolor{darkmidnightblue}{rgb}{0.0, 0.2, 0.4}
\definecolor{darkpowderblue}{rgb}{0.0, 0.2, 0.6}
\definecolor{dukeblue}{rgb}{0.0, 0.0, 0.61}

\hypersetup{
    colorlinks = true,
    citecolor= midnightblue,
    urlcolor= black,
    breaklinks=true,
    linkcolor = dukeblue,
    linkbordercolor = {white},
}
\usepackage{natbib}[round,sort]
\def\bvartheta{\boldsymbol\vartheta}  
\def\btheta{\boldsymbol\theta}
\def\bxi{\boldsymbol\xi}
\def\bzeta{\boldsymbol\zeta}
\def\bL{\boldsymbol L}
\def\bW{\boldsymbol W}

\def\bG{\boldsymbol G}
\def\wass{{\sf W}}
\def\bV{\boldsymbol V}
\def\bZ{\boldsymbol Z}
\def\bY{\boldsymbol Y}
\def\bg{\boldsymbol g}
\def\bv{\boldsymbol v}
\def\by{\boldsymbol y}

\newtheorem{theorem}{Theorem}
\newtheorem{assumption}{Assumption}
\newtheorem{proposition}{Proposition}

\newtheorem{lemma}{Lemma}
\newtheorem*{lemma*}{Lemma}
\newtheorem{corollary}{Corollary}

\def\bB{\boldsymbol{B}}

\def\bfI{\mathbf I}
\def\bfH{\mathbf H}
\def\bfA{\mathbf A}
\def\bfB{\mathbf B}

\def\bfC{\mathbf C}

\def\Ltwo{\mathbb L_2}
\def\RR{\mathbb R}
\def\NN{\mathbb N}
\def\rmd{{\rm d}}
\def\E{\mathbb E}
\newcommand{\grad}{\nabla}
\newcommand{\norm}[1]{\left\| #1 \right\|}

\newcommand{\brr}[1]{\left( #1 \right)}   

\definecolor{darkmidnightblue}{HTML}{003366}    
\definecolor{midnightblue}{HTML}{0059b3}
\definecolor{chromered}{HTML}{f14233}

\usepackage[colorinlistoftodos,bordercolor=black,backgroundcolor= midnightblue!20,linecolor=black,textsize=scriptsize]{todonotes}

\newcommand{\avetik}[1]{\todo[inline]{\textbf{Avetik:} #1}}

\title{Langevin Monte Carlo for strongly log-concave distributions: Randomized midpoint revisited}

\author{%
  Lu Yu\\
  CREST, ENSAE, IP Paris\\
  \texttt{lu.yu@ensae.fr} \\
  \And 
  Avetik Karagulyan\\
  KAUST\\
  \texttt{avetik.karagulyan@kaust.edu.sa} \\
  \And
  Arnak Dalalyan \\
  CREST, ENSAE, IP Paris\\
}

\begin{document}

\maketitle

\begin{abstract}
  We revisit the problem of sampling from a target
  distribution  that has a smooth strongly log-concave 
  density everywhere in $\mathbb R^p$. In this context, 
  if no additional density information is available, 
  the randomized midpoint discretization for the 
  kinetic Langevin diffusion is known to be the most 
  scalable method in high dimensions with large 
  condition numbers. Our main result is a nonasymptotic 
  and easy to compute upper bound on the $\wass_2$-error 
  of this method. To provide a more thorough 
  explanation of our method for establishing the 
  computable upper bound, we conduct an analysis 
  of the midpoint discretization for the vanilla 
  Langevin process. This analysis helps to clarify 
  the underlying principles and provides valuable 
  insights that we use to establish an improved 
  upper bound for the kinetic Langevin process with 
  the midpoint discretization. Furthermore, by 
  applying these techniques we establish new 
  guarantees for the kinetic Langevin process 
  with Euler discretization, which have a better 
  dependence on the condition number than existing 
  upper bounds.
\end{abstract}
\addtocontents{toc}{\protect\setcounter{tocdepth}{0}}

\section{Introduction}

The task of sampling from target distributions with smooth, strongly log-concave densities has been a long-standing challenge in various fields such as statistics, machine learning, and computational physics~\citep{andrieu2003introduction,krauth2006statistical,andrieu2010particle}. 
Over the years, researchers have developed several algorithms to tackle this problem, and one prominent approach are Langevin algorithms~\citep{rogers2000diffusions,oksendal2013stochastic,robert1999monte}.
Langevin algorithms leverage the Langevin equation to design efficient and effective sampling algorithms. 
These methods generate a Markov chain by iteratively updating the position of a particle based on the Langevin equation. 
By simulating the particle's motion over time, these algorithms explore the target distribution and eventually converge to samples that approximate the desired distribution~\citep{robert1999monte}.

The canonical sampling algorithm, Langevin Monte Carlo (LMC)~\citep{RobertsTweedie96,Dalalyan14,durmus2017,erdogdu2021convergence,mousavi2023towards,raginsky2017non,Erdogdu18,mou2022improved,erdogdu2022convergence}, is a Markov chain Monte Carlo (MCMC) method that simulates the dynamics of a fictitious particle moving through a potential energy landscape defined by the target distribution. Formally, it is the Euler-Maruyama discretization of an SDE known as Langevin diffusion. 
The underlying idea can be traced back to the early 20th century when Paul Langevin introduced a stochastic differential equation (SDE) to describe the motion of a particle in a fluid \citep{langevin1908theorie}. 
This SDE, known as Langevin diffusion, combines deterministic and random components to model the particle's behavior under the influence of both a deterministic force and random noise.

One popular variant of Langevin Monte Carlo is based on discretizing the kinetic Langevin diffusion, which introduces a friction term to control the exploration-exploitation trade-off during the sampling process \citep{einstein1905molekularkinetischen,von1906kinetischen}.  
According to \cite{Nelson}, Langevin diffusion is the rescaled limit of the kinetic Langevin diffusion. 
Its ergodicity and mixing-time properties are studied in~\cite{eberle2019,dalalyan_riou_2018}.
 Euler-Maruyama time discretization of this SDE, called kinetic Langevin Monte Carlo (KLMC), is prevalent in the sampling literature~\citep{cheng2018underdamped,dalalyan_riou_2018,shen2019randomized,ma2019there,zhang2023improved}.

The randomized midpoint discretization method, as an alternative to the Euler-Maruyama scheme for KLMC, is proposed by \cite{shen2019randomized}.
They demonstrate the superior performance of this method in terms of both tolerance and condition number dependency.
More recently, \cite{he2020ergodicity} analyze probabilistic properties of the randomized midpoint discretization method for (kinetic) Langevin diffusion.
In this work,  we conduct a comprehensive and thorough analysis of the randomized midpoint discretization scheme for the kinetic Langevin diffusion under strongly log-concavity, and establish improved non-asymptotic and computable upper bounds on the discretization error for this method.
Towards that, we make the following contributions.
\begin{itemize}
\setlength\itemsep{0.05em}
    \item To lay the groundwork for our analysis, we initially delve into the midpoint discretization technique applied to the vanilla Langevin process.
    We establish in~\Cref{thm:rlmc} the convergence guarantees for
RLMC in $\wass_2$-distance. 
    These guarantees are competitive with the best available results for LMC, and could be leveraged to derive an improved upper bound specifically tailored for RKLMC.
    \item We further extend these techniques to RKLMC, and provide the corresponding convergence guarantees in $\wass_2$-distance in Theorem~\ref{thm:rklmc}.
    Compared to the previous works, our bound \textbf{a)} contains explicit and small constants, \textbf{b)} does not require the initialization to be at minimizer of the potential, \textbf{c)} and is free from the linear dependence on the sample size.
    
    \item Employing the same techniques, we finally examine the convergence behavior of the KLMC algorithm with the Euler-Maruyama discretization.
    In \Cref{thm:klmc1}, we provide an upper 
    bound on the accuracy of this scheme in $\wass_2$-distance with improved 
    dependence on the condition number.
\end{itemize}

\textbf{Notation.}
Denote the $p$-dimensional Euclidean space by $\RR^p$.
The letter $\btheta$ denotes the deterministic vector and its calligraphic counterpart 
$\bvartheta$ denotes the random vector.
We use $\bfI_p$ and $\mathbf 0_p$ to denote, respectively, the $p \times p$ identity and zero matrices.  
Define the relations $\bfA
\preccurlyeq\bfB$ and $\bfB\succcurlyeq \bfA$ 
for two symmetric $p \times p$ matrices $\bfA$ 
and $\bfB$ to mean that $\bfB - \bfA$ is 
semi-definite positive. 
The gradient and the Hessian of a function $f:\RR^p \rightarrow \RR$ are denoted by $\nabla f$  and 
 $\nabla^2 f$, respectively.
Given any pair of measures $\mu$ and $\nu$ defined on $(\RR^p,\mathcal{B}(\RR^p))$,
the Wasserstein-2 distance between $\mu$ and $\nu$ is defined as
\begin{equation}
\wass_2(\mu,\nu) = \Big(\inf_{\varrho\in \Gamma(\mu,\nu)} \int_{\RR^p\times \RR^p}
\|\btheta -\btheta'\|_2^2 \,\rmd\varrho(\btheta,\btheta')\Big)^{1/2},
\end{equation}
where the infimum is taken over all joint distributions $\varrho$ that have $\mu$ and $\nu$ as marginals. 

\section{Understanding the randomized midpoint 
discretization: the vanilla Langevin diffusion}

The goal is to sample a random vector in 
$\mathbb R^p$ according to a given distribution $\pi$
of the form
\begin{align}
    \pi(\btheta)\propto \exp\{-f(\btheta)\},\qquad 
    \btheta\in\mathbb R^p,
\end{align}
with a function $f:\RR^p\to\RR$, referred to as the potential.
Throughout the paper, we assume that the 
potential function $f$ is $M$-smooth 
and $m$-strongly convex for some constants 
$0<m\leqslant M<\infty$. 


\begin{assumption}\label{asm:A-scgl}
   The function $f: \mathbb R^p\to\mathbb R$ 
   is twice differentiable, and its Hessian 
   matrix $\nabla^2 f$ satisfies 
\begin{align}
    m\bfI_p\preccurlyeq \nabla^2 f(\btheta)
    \preccurlyeq M\bfI_p,\qquad \forall 
    \btheta\in\mathbb R^p.
\end{align}
\end{assumption}
Let $\bvartheta_0$ be a random vector drawn from a 
distribution $\nu$ on $\mathbb R^p$ and let $\bW 
=(\bW_t: t\geqslant 0)$ be a $p$-dimensional 
Brownian motion independent of $\bvartheta_0$. 
Using the potential $f$, the random variable 
$\bvartheta_0$ and the process $\bW$, one can define 
the stochastic differential equation
\begin{align}\label{eq:LD}
    d\bL_t^{\sf LD} = -\nabla f(\bL_t^{\sf LD})\,dt 
    + \sqrt{2}\,\rmd\bW_t,
    \qquad t\geqslant 0,\qquad \bL^{\sf LD}_0=\bvartheta_0.
\end{align}
This equation has a unique strong solution, which 
is a continuous-time Markov process, termed Langevin 
diffusion. Under some further assumptions on $f$, 
such as strong convexity or dissipativity, the 
Langevin diffusion is ergodic, geometrically
mixing and has $\pi$ as its unique invariant 
distribution \citep{bhattacharya1978}. 
Furthermore, the mixing properties
 of this process can be quantified. For instance, 
if $\pi$ satisfies the Poincar\'e inequality with
constant $C_{\textsf{P}}$, then (see e.g. \cite{chewi2020exponential}) the distribution
$\nu_t^{\sf LD}$ of $\bL^{\sf LD}_t$ satisfies
\begin{align}
    \wass_2(\nu_t^{\textsf{LD}},\pi) \leqslant e^{-t/
    C_{\textsf{P}}} 
    \sqrt{2\smash[b]{C_{\sf P}\chi{}^2(\nu\|\pi)}},\qquad \forall t
    \geqslant 0. 
\end{align}
These results suggest that we can sample from the 
distribution $\pi$ by using a suitable discretization 
of the Langevin diffusion. The Langevin Monte Carlo 
(LMC) method is based on this idea, combining the 
aforementioned considerations with the Euler 
discretization. Specifically, for small values of 
$h \geqslant 0$ and $\Delta_h\bW_t = \bW_{t+h} - 
\bW_t$, the following approximation holds
\begin{align}
    \bL^{\sf LD}_{t+h} & = \bL^{\sf LD}_t- \int_0^h  
    \nabla f(\bL_{t+s}^{\sf LD})\,\rmd s + \sqrt{2}\;\Delta_h\bW_t \approx \bL_t^{\sf LD} - h  
    \nabla f(\bL_t^{\sf LD}) + \sqrt{2}\;\Delta_{h}
    \bW_t. 
\end{align}
By repeatedly applying this approximation with a small 
step-size $h$, we can construct a Markov chain 
$(\bvartheta^{\sf LMC}_k:k\in\mathbb N)$ that converges 
to the target distribution $\pi$ as $h$ goes to zero. 
More precisely, $\bvartheta^{\sf LMC}_k\approx \bL^{
\sf LD}_{kh}$, for $k\in\mathbb N$, is given by
\begin{align}
    \bvartheta_{k+1}^{\sf LMC} = \bvartheta_k^{\sf LMC} 
    -h\nabla f(\bvartheta_k^{\sf LMC})
    +\sqrt{2}\,(\bW_{(k+1)h}-\bW_{kh}).
\end{align}
This method is computationally efficient and has been 
widely used in statistics and machine learning for 
sampling from high-dimensional distributions \citep{gal2016dropout,izmailov2020subspace,izmailov2021bayesian}. To 
assess the discretization error, consider the case 
where $\bL^{\sf LD}_0$ is drawn from the invariant 
distribution $\pi$ and note that 
\begin{align}
    \bL^{\sf LD}_{(k+1)h} - \bvartheta^{\sf LMC}_{k+1} 
    & = \bL^{\sf LD}_{kh} -\bvartheta^{\sf LMC}_k - 
    \int_0^h \nabla f(\bL^{\sf LD}_{kh+s}) \,\rmd s +
    h\nabla f(\bvartheta^{\sf LMC}_k)\\
    & = \bL^{\sf LD}_{kh} - \bvartheta^{\sf LMC}_k -
    h\big(\nabla f (\bL^{\sf LD}_{kh}) -\nabla f( 
    \bvartheta^{\sf LMC}_k)\big) - \bzeta_k,
    \label{eq:err1}
\end{align}
where $\bzeta_k = \int_0^h \big(\nabla f(\bL^{\sf 
LD}_{kh+s}) - \nabla f(\bL^{\sf LD}_{kh})\big)\,\rmd s$ 
is a zero-mean random ``noise'' vector. Previous work 
on LMC demonstrated that the squared $\mathbb L_2$ norm 
of $\bzeta_k$ is of order $M^2h^3p$, 
whereas the term $ \bL^{\sf LD}_{kh} - \bvartheta^{\sf 
LMC}_k-h\big(\nabla f(\bL^{\sf LD}_{kh}) - \nabla f 
(\bvartheta^{\sf LMC}_k) \big)$ satisfies the contraction
inequality
\begin{align}
    \big\| \bL^{\sf LD}_{kh} - \bvartheta^{\sf LMC}_{k} 
    - h\big(\nabla f (\bL^{\sf LD}_{kh}) -\nabla f 
    (\bvartheta^{\sf LMC}_{k})\big)\big \|^2_{\mathbb 
    L_2}\leqslant (1-mh)^2\|\bL^{\sf LD}_{kh} 
    -\bvartheta^{\sf LMC}_{k}\|_{\Ltwo}^2.\label{eq:err2}
\end{align}
If we denote by $r_k$ the correlation between $\bzeta_k$
and $\bL^{\sf LD}_{kh} - \bvartheta^{\sf LMC}_{k}$, and 
by $\text{Err}_k$ the error $\|\bL^{\sf LD}_{kh} -
\bvartheta^{\sf LMC}_{k} \|_{\Ltwo}$, we infer 
from \eqref{eq:err1} and \eqref{eq:err2} that
\begin{align}
    \text{Err}_{k+1}^2 &\leqslant (1-mh)^2\text{Err}_k^2 
    + C  Mh r_k \text{Err}_k \sqrt{hp} + C M^2 h^3 p,
\end{align}
for some universal constant $C$. If we were able
to check that $r_k$ is small enough so that the 
second term of the right-hand side can be neglected,
we would get $\text{Err}_{k+1}^2 \leqslant (1-mh)^2 
\text{Err}_k^2 + C M^2 h^3 p$, which would eventually
lead to $\text{Err}_{k+1}^2 \leqslant (1-mh)^{2k}
\text{Err}_1^2 + C M^2 h^2 (p/m)$. This would amount
to
\begin{align}
    |r_k|\ll 1\qquad \Longrightarrow\qquad\text{Err}_{k+1} 
    \leqslant 
    (1-mh)^{k} \,\text{Err}_1 + C M h \sqrt{p/m}.
    \label{eq:err3}
\end{align}
Unfortunately, without any additional conditions on
$f$, the correlation $r_k$ cannot be shown to be
small, and one can only deduce from \eqref{eq:err2}
that $\text{Err}_{k+1}\leqslant (1-mh)\text{Err}_k 
+ C M h \sqrt{ph}$, which eventually yields
\begin{align}
    |r_k|\not\ll 1\qquad\Longrightarrow\qquad 
    \text{Err}_{k+1} \leqslant (1-mh)^{k} \,
    \text{Err}_1 + C (M/m)  \sqrt{ph}.\label{eq:err4}
\end{align}
This inequality is established under the standard
assumption $Mh\leqslant 2$, which implies that the 
last term in \eqref{eq:err3} is 
significantly smaller than \eqref{eq:err4}. To 
get such an error deflation, we need the correlations
$r_k$ to be small. While this is not guaranteed for the
Euler discretization, we will see that the randomized 
midpoint method allows us to achieve such a reduction.

Let $U$ be a random variable uniformly distributed in
$[0,1]$ and independent of the Brownian motion $\bW$. 
The randomized midpoint method exploits the approximation
\begin{align}
    \bL^{\sf LD}_{t+h}=\bL^{\sf LD}_t - \int_0^h 
    \nabla f(\bL^{\sf LD}_{t+s})\, \rmd s + \sqrt{2}\;\Delta_h\bW_t \approx \bL^{\sf LD}_t 
    - h \nabla f(\bL^{\sf LD}_{t+hU})+
    \sqrt{2}\;\Delta_h\bW_t.
\end{align}
The noise counterpart of $\bzeta_k$ in this case is
$\bzeta_k^{\sf R}= \int_0^h \nabla f(\bL^{\sf LD}_{t+s})
\,\rmd s-\nabla f(\bL^{\sf LD}_{t+Uh})$. It is clearly 
centered and uncorrelated with all the random
vectors independent of $U$ such as $\bL^{\sf LD}_{kh}$, 
$\bvartheta^{\sf LMC}_k$ and the gradient of $f$ 
evaluated at these points. 

The explanation above provides the intuition of 
the randomized midpoint method, and a hint to
why it is preferable to the Euler discretization,
but it cannot be taken as a formal definition of
the method. The formal definition of  the 
randomized midpoint method for Langevin Monte 
Carlo (RLMC) is defined as follows: at each 
iteration $k=1,2,\ldots$, 
\vspace{-5pt}
\begin{enumerate}
\setlength\itemsep{0.05em}
    \item we randomly, and independently of all 
    the variables generated during the previous 
    steps, generate a pair of random vectors
    $(\bxi'_k,\bxi''_k)$ and a random variable 
    $U_k$ such that
    \begin{itemize}
        \item $U_k$ is uniformly distributed in 
        $[0,1]$ and independent of $(\bxi'_k, 
        \bxi''_k)$,
        \item $(\bxi'_k,\bxi''_k)$ are 
        independent $\mathcal N_p(0,\bfI_p)$.
    \end{itemize} 
    \item 
    we set $\bxi_k = \sqrt{U_k}\,\bxi_k' + 
    \sqrt{1-U_k}\bxi''_k$ and define the 
    $(k+1)$th iterate $\bvartheta^{\sf RLMC}$ by 
    \begin{align} 
        \bvartheta_{k+U}^{\sf RLMC} &=  
        \bvartheta_k^{\sf RLMC} - hU_k \nabla
        f(\bvartheta_k^{\sf RLMC}) + \sqrt{2h U_k}
        \, \bxi'_{k} \label{meth-mid1}\\
        \bvartheta_{k+1}^{\sf RLMC} & = 
        \bvartheta_k^{\sf RLMC}  - h\nabla f 
        (\bvartheta_{k+U}^{\sf RLMC}) + \sqrt{2h}
        \, \bxi_{k}.\label{meth-mid2}
    \end{align}
\end{enumerate}
With a small step-size $h$ and a large number 
of iterations $n$, the distribution of 
$\bvartheta_n^{\sf RLMC}$ can closely approximate 
the target distribution $\pi$. In a smooth and 
strongly convex setting, it is even possible to 
obtain a reliable estimate of the sampling error, 
as demonstrated in the following theorem (the 
proof is included in the supplementary material).

If the step-size $h$ is small and the number of
iterations $n$ is large, the distribution of 
$\bvartheta_n^{\sf RLMC}$ is a close to the target
$\pi$. Interestingly, in the smooth and strongly
convex setting it is possible to get a good 
evaluation of the error of sampling as shown
in the next theorem (the proof is deferred to
the supplementary material).  

\begin{theorem}
\label{thm:rlmc}
    Assume the function $f:\RR^p\to \RR$ satisfies  Assumption~\ref{asm:A-scgl}. 
    Let 
    $h$ be such that $Mh + \sqrt{\kappa}\,(Mh)^{3/2} 
    \leqslant 1/4$ with $\kappa = M/m$. Then, 
    every $n\geqslant 1$, the distribution 
    $\nu_n^{\sf RLMC}$ of ~$\bvartheta_n^{\sf RLMC}$ 
    satisfies 
    \begin{align}\label{ineq:RLMC}
        \wass_2(\nu_n^{\sf RLMC}, \pi) &\leqslant  
        1.11 e^{-mnh/2} \wass_2(\nu_0,\pi)   + 
        \big(2.4 \sqrt{\kappa Mh}  + 1.77\big) 
        Mh\sqrt{p/m}.
    \end{align}
\end{theorem}
Prior to discussing the relation of the above
error estimate to those available in the literature,
let us state a consequence of it.
\begin{corollary}\label{cor:1}
    Let $\varepsilon \in(0,1)$ be a small number. 
    If we choose $h>0$ and $n\in\mathbb N$ so that 
    \begin{align}
        Mh =  \frac{\varepsilon}{1.5 + (6.5\kappa 
        \varepsilon)^{1/3}},
    \quad
    \text{and} 
    \quad
        n\geqslant \bigg(\frac{3\kappa}{\varepsilon} 
        + \frac{3.8\kappa^{4/3}}{\varepsilon^{2/3}}
        \bigg)\bigg(
        \log(20/\varepsilon) + \frac12\log\Big(
        \frac{m}{p}\wass_2^2(\nu_0,\pi)\Big)\bigg)
    \end{align}
    then\footnote{This follows from the fact that
    $(6\kappa/{\varepsilon}) + 4.2\kappa^{4/3}/{
    \varepsilon^{2/3}}\leqslant 2\kappa / (Mh) = 2/mh$.} 
    we have $\wass_2(\nu_n^{\sf RLMC}, 
    \pi) \leqslant \varepsilon\sqrt{p/m}$.
\end{corollary}

Our results can be compared to the best available 
results for Langevin Monte Carlo (LMC) under Assumption~\ref{asm:A-scgl}
\cite[Eq. 22]{durmus2019analysis}. We recall that 
LMC is defined by a recursive relation of the
same form as \eqref{meth-mid2}, with the 
only difference that $\nabla f(\bvartheta_{k+U})$ 
is replaced by $\nabla f(\bvartheta_k)$. The tightest 
known bound for LMC is given by 
\begin{align}
    \wass_2(\nu_n^{\sf LMC}, \pi) &\leqslant
    (1 - mh)^{- n/2} \wass_2(\nu_0,\pi) +
    \sqrt{2Mh p/m},
\end{align}
with $Mh\leqslant 1$. By choosing $2Mh = (19/20)^2 
\varepsilon^2$ and 
\begin{align}
    n\geqslant 2.22(\kappa/\varepsilon^2)
    \big\{\log(20/\varepsilon) + \textstyle{\frac12}
    \log \big({\textstyle\frac{m}{p}} \wass_2^2 
    (\nu_0,\pi)\big) \big\},
\end{align}
we can ensure that $\wass_2(\nu_n^{\sf LMC}, \pi)
\leqslant\varepsilon\sqrt{p/m}$. Therefore, our 
bounds imply superior results for RLMC in the regime
of $\kappa$ of smaller order than $\varepsilon^{-4}$.

To the best of our knowledge, the first results on the 
error analysis of RLMC have been obtained in 
\citep{he2020ergodicity}. They derived an upper bound on 
the discretization error (the second term on the 
right-hand side of \eqref{ineq:RLMC}) under the assumption 
that the initial point of the algorithm is the minimizer 
of the potential function $f$. Their bound takes the 
form $C (\sqrt{\kappa M h} +1)Mh\sqrt{p/m}\times\sqrt{mnh}$, 
where $C$ is a universal but unspecified constant. 
Compared to our bound, the one obtained in 
\cite{he2020ergodicity} has an additional factor 
$\sqrt{mnh}$. While this factor may not be very harmful 
in the case of geometric ergodicity where the number of 
iterations $n$ is chosen such that $nmh$ goes to 
infinity at the logarithmic rate $\log(1/\varepsilon)$, 
removing it can be an important step toward extending 
these results to potentials that are not strongly convex.

\section{Randomized midpoint method for the 
kinetic Langevin diffusion}

The randomized midpoint method, introduced and 
studied in \citep{shen2019randomized}, aims at providing
a discretization of the kinetic Langevin process
that reduces the bias of sampling as compared to
more conventional discretizations. Recall that
the kinetic Langevin process $\bL^{\textup{\sf KLD}}$
is a solution to a second-order stochastic
differential equation that can be informally 
written as
\begin{align}\label{KLD:1}
  {\textstyle\frac1{\gamma}}\ddot\bL_t^{\sf KLD} + 
  \dot\bL_t^{\sf KLD} = -\nabla f(\bL_t^{\sf KLD}) 
  + \sqrt{2}\,\dot\bW_t,
\end{align}
with initial conditions $\bL_0^{\sf KLD} = 
\bvartheta_0$ and $\dot\bL_0^{\sf KLD} = \bv_0$. 
In \eqref{KLD:1}, $\gamma>0$, $\bW$ is a standard
$p$-dimensional Brownian motion and dots are used
to design derivatives with respect to time $t
\geqslant 0$. This can be formalized using It\^o's
calculus and introducing the velocity field 
$\bV^{\sf KLD}$ so that the joint process 
$(\bL^{\sf KLD}, \bV^{\sf KLD})$ satisfies
\begin{align}\label{KLD:2}
    \rmd\bL^{\sf KLD}_t = \bV_t^{\sf KLD}\,\rmd t;
    \quad  
    \tfrac1{\gamma}\rmd\bV^{\sf KLD}_t = -\big(\bV_t^{\sf KLD} + 
    \nabla f(\bL_t^{\sf KLD})\big)\,\rmd t + 
    \sqrt{2}\, \rmd \bW_t.
\end{align}
Similar to the vanilla Langevin diffusion \eqref{eq:LD}, 
the kinetic Langevin diffusion $(\bL^{\sf KLD}, \bV^{\sf
KLD})$ is a Markov process that exhibits ergodic 
properties when the potential $f$ is strongly 
convex (see \citep{eberle2019} and references 
therein). The invariant density of this process 
is given by
\begin{align}
    p_*(\btheta_,\bv) \propto \exp\{-f(\btheta) - {\textstyle\frac1{2\gamma}}\|\bv\|^2\}, \qquad
    \text{for all}\quad \btheta,\bv\in\mathbb R^p.
\end{align}
Note that the marginal of $p_*$ corresponds to 
$\btheta$ coincides with the target density $\pi$. 
However, unlike the vanilla Langevin diffusion, 
the kinetic Langevin is not reversible. It is 
interesting to note that the distribution of the 
process $\bL^{\sf KLD}$ approaches that of the 
vanilla Langevin process as $\gamma$ approaches 
infinity (see e.g. \citep{Nelson}). Therefore, 
$\bL^{\sf LD}$ and 
$\bL^{\sf KLD}$ are often referred to as 
overdamped and underdamped Langevin processes, 
respectively (where increasing the friction 
parameter $\gamma$ is characterized as damping). 

The kinetic Langevin diffusion $\bL^{\sf KLD}$ is particularly attractive for sampling because its distribution $\nu_t^{\sf KLD}$ converges to the invariant distribution exponentially fast. This is especially true for strongly convex potentials, as proven in\footnote{For the sake of the self-containedness of this paper, we reproduce the proof of this inequality in \Cref{prop:klmc_contr} deferred to the Appendix.} \cite[Prop.~1]{dalalyan_riou_2018}, where it is shown that the following inequality holds:
\begin{align}
    \wass_2\bigg(\bfC
    \begin{bmatrix}
        \bV_t^{\sf KLD}\\[1pt]
        \bL_t^{\sf KLD}
    \end{bmatrix}
    ,
    \bfC
    \begin{bmatrix}
        \bv\\[1pt]
        \bvartheta
    \end{bmatrix}
    \bigg)
    \leqslant e^{-mt}
    \wass_2\bigg(\bfC
    \begin{bmatrix}
        \bV_0\\[1pt]
        \bL_0
    \end{bmatrix}
    , \bfC
    \begin{bmatrix}
        \bv\\[1pt]
        \bvartheta
    \end{bmatrix}
    \bigg),\quad 
    \bfC = \begin{bmatrix}
        \bfI_p & \mathbf 0_p\\[1pt]
        \bfI_p & \gamma \bfI_p
    \end{bmatrix}
\end{align}
for every $t\geqslant 0$, provided that $\gamma
\geqslant m+M$.

To discretize this continuous-time process and make it applicable to the sampling problem, \cite{shen2019randomized} proposed the following procedure: at each iteration $k=1,2,\ldots$,
\begin{enumerate}
    \item randomly, and independently of all the variables
    generated at the previous steps, generate random 
    vectors $(\bxi_k',\bxi''_k,\bxi_k''')$ and a random
    variable $U_k$ such that
    \begin{itemize}
    \setlength\itemsep{0.05em}
        \item $U_k$ is uniformly distributed in $[0,1]$,
        \item conditionally to $U_k = u$, $(\bxi_k',\bxi''_k,\bxi'''_k)$ has the same
        joint distribution as 
            $\big(\bB_u - e^{-\gamma h u} \bG_u,
                \bB_1 -e^{-\gamma h}\bG_1,
                \gamma e^{-\gamma h }\bG_1
            \big)$,
        where $\bB$ is a $p$-dimensional Brownian motion
        and $\bG_t = \int_0^{t} e^{\gamma h s}\,\rmd 
        \bB_s$.
    \end{itemize}
    \item set $\psi(x) = (1-e^{-x})/x$ and define the 
    $(k+1)$th iterate of $\bvartheta^{\sf RKLMC}$ by
    \begin{align}
        \bvartheta_{k+U} &=\bvartheta_k+  
        U h \psi(\gamma U h) \bv_k - {U h} \big(1 - \psi 
        (\gamma Uh)\big)\nabla f(\bvartheta_k) + 
        \sqrt{2h}\,\bxi'_k\\
        \bvartheta_{k+1} &= \bvartheta_k +  h
        \psi(\gamma h)\bv_k -  \gamma h^2( 1 - U)
        \psi\big(\gamma h(1-U)\big) \nabla f 
        (\bvartheta_{k+U}) + \sqrt{2h} \bxi''_k\\
        \bv_{k+1} &= e^{-\gamma h}\bv_k - {\gamma} h
        e^{- \gamma h(1 -U)} \nabla f(\bvartheta_{k+U}) 
        +  \sqrt{2h}\,\bxi'''_k.
    \end{align}
\end{enumerate}
Although the sequence $(\bv_k^{\sf RKLMC}, \bvartheta_k^{\sf RKLMC})$ approximates $(\bV_{kh}^{\sf KLD}, \bL_{kh}^{\sf KLD})$, it is not immediately apparent. The supplementary material clarifies this point. We state now the main result 
of this paper, providing a simple upper bound for the error 
of the RKLMC algorithm.

\begin{theorem}
\label{thm:rklmc}
Assume the function $f:\RR^p\to \RR$ satisfies  
Assumption~\ref{asm:A-scgl}. Choose
$\gamma$ and $h$ so that $\gamma \geqslant 5M$ 
and  $\gamma h \leqslant 0.1\kappa^{-1/6}$, where $\kappa = M/m$. 
Assume that 
$\bvartheta_0$ is independent of $\bv_0$ and that
$\bv_0\sim\mathcal N_p(0,\gamma\bfI_p)$. Then, for any 
$n \geqslant 1$,
the distribution $\nu_n^{\textup{\sf RKLMC}}$ of 
$\bvartheta_n^{\textup{\sf RKLMC}}$ satisfies
\begin{alignat}{2}
    \wass_2(\nu_n^{\textup{\sf RKLMC}},\pi) &
    \leqslant 1.6\varrho^n  \wass_2(\nu_0,\pi) 
    &\,+\,& 0.1 \sqrt{\varrho^n\mathbb E[f(\bvartheta_0) 
    - f(\btheta_*)]/m}\\
    & &\,+\,&  0.2(\gamma h)^3\sqrt{\kappa p/m} 
     + 10(\gamma h)^{3/2}  \sqrt{p/m}\,,
\end{alignat}
where $\varrho  = \exp(-mh),$ and $\btheta_*=\arg\min_{\btheta\in\RR^p} f(\btheta)$.
\end{theorem}
This result has several strengths and limitations, 
which are discussed below, after the corollary providing
the number of required iterations to 
attain a predetermined level of accuracy.

\begin{corollary}
    Let $\varepsilon\in(0,1)$ be a small constant. 
    If $\gamma = 5M$, $\bvartheta_0 = \btheta_*$ and 
    we choose $h>0$ and $n\in\mathbb N$ so that
    \begin{align}
        \gamma h =  \frac{\varepsilon^{2/3}}{
        5 + 0.6(\varepsilon^2\kappa)^{1/6} },\quad
    \text{and} 
    \quad
        n\geqslant \kappa{\varepsilon^{-2/3}}
        \big(25 + 3(\varepsilon^2\kappa)^{1/6}\big)
        \log(20/\varepsilon) \,,   
    \end{align}
    then we have $\wass_2(\nu_n^{\sf RKLMC},\pi)\leqslant 
    \varepsilon\sqrt{p/m}$.  
\end{corollary}

The corollary presented above gives the best-known convergence rate for the number of gradient evaluations required to achieve a prescribed error level in the case of a gradient Lipschitz potential, without any additional assumptions on its structure or smoothness. This rate, $\kappa\varepsilon^{-2/3}(1+ (\varepsilon^2\kappa)^{1/6})$, was first discovered by \cite{shen2019randomized} (see also \citep{he2020ergodicity}). One of the main strengths of our result
in \Cref{thm:rklmc} is that it removes a factor $nmh$ 
from the discretization error, which was present in 
the previous upper bounds of the sampling error. 
Furthermore, our bound contains only small and explicit constants. Finally, our result does not require the RKLMC algorithm to be initialized at the minimizer of the potential, which is important for extending the method to non-convex potentials.


On the downside, the condition $\gamma\geqslant 5M$ 
is stronger than the corresponding conditions used in 
prior work for kinetic Langevin Monte Carlo (without
randomization). Indeed, these prior results generally
require $\gamma\geqslant 2M$. Having a proof of 
\Cref{thm:rklmc} that reduces the factor 5 in 
$\gamma\geqslant 5M$ would lead to significant 
savings in running time. A second limitation of our
result is its dependence on the initial error. If 
the algorithm is not initialized at the minimum of $f$,
our bound would imply that the number of iterations to 
achieve an accuracy $\varepsilon \sqrt{p/m}$ scales 
polynomially in the initial error $\wass_2(\nu_0,\pi)$.

While the proof of this theorem is deferred to the supplementary material, we can outline the main argument that allowed us to remove the factor $nmh$ from the error bound. To convey the main idea, let us consider three positive sequences ${a_n}$, ${b_n}$, ${c_n}$ satisfying, for every $n\in\mathbb N$,
\begin{align}
a_{n+1} &\leqslant (1 - \alpha) a_n + b_n \label{an}\\
c_{n+1} &\leqslant c_n - b_n + {\sf C}, \label{cn}
\end{align}
with some $\alpha\in(0,1)$ and ${\sf C}>0$. Using the standard telescoping sums argument, frequently employed for proving the convergence of convex optimization algorithms, one can infer from \eqref{cn} that
\begin{align}\label{bn}
\sum_{k=0}^n\nolimits b_n \leqslant c_0-c_{n+1} + n{\sf C}\leqslant c_0+ n{\sf C}.
\end{align}
On the other hand, it follows from \eqref{an} that
\begin{align}\label{an1}
a_{n+1} &\leqslant (1-\alpha)^{n+1} a_0 + \sum_{k=0
}^n\nolimits (1-\alpha)^{n-k}b_k.
\end{align}
Upper bounding $(1-\alpha)^{n-k}$ by one, and using 
\eqref{bn}, we arrive at
\begin{align}\label{an2}
a_{n+1}\leqslant (1-\alpha)^{n+1} a_0 + c_0 + n{\sf C}.
\end{align}
This type of argument, used in previous papers on 
RKLMC, is sub-optimal and leads to the extra factor 
$nmh$. A tighter bound can be obtained by replacing 
the telescoping sum argument by the summation by parts. 
More precisely, one can check that \eqref{cn} and 
\eqref{an1} yield
\begin{align}
    a_{n+1} &\leqslant (1-\alpha)^{n+1} a_0 + \sum_{
    k=0}^n\nolimits (1-\alpha)^{n-k} (c_k - c_{k+1}) 
    + {\sf C}\sum_{k=0}^n\nolimits (1-\alpha)^{n-k}\\
    &\leqslant (1-\alpha)^{n+1} a_0 + (1-\alpha)^nc_0 
    + \alpha\sum_{k=0}^n\nolimits (1-\alpha)^{n-k} c_k 
    + \frac{\sf C}{\alpha}. \label{an3}
\end{align}
The upper bound provided by \eqref{an3} has two 
advantages as compared to \eqref{an2}: the term 
$n\sf C$ is replaced by ${\sf C}/\alpha$, which 
is generally smaller, and the dependence on the 
initial value is $(1-\alpha)^nc_0$ instead of 
$c_0$. This comes also with a challenge consisting 
in upper bounding the sum present in the right-hand 
side of \eqref{an3}, which we managed to overcome 
using the strong convexity (or, more precisely, 
the Polyak-Lojasiewicz condition). The full details 
are deferred to the supplementary material.

\section{Improved error bound for the kinetic 
Langevin with Euler discretization}

The proof techniques presented in the previous 
section can be used to derive an upper bound on the
error of the kinetic Langevin Monte Carlo (KLMC) 
algorithm. KLMC is a discretized version of KLD 
\eqref{KLD:2}, where the term $\nabla f(\bL_t)$ 
is replaced by $\nabla f(\bL_{kh})$ on each interval 
$[kh, (k+1)h)$. The resulting error bound, given 
in the following theorem, exhibits a better 
dependence on $\kappa$ than previously established 
bounds.

\begin{theorem}
\label{thm:klmc1}
Let $f:\mathbb R^p\to\mathbb R$ satisfy $m\bfI_p
\preccurlyeq \nabla^2 f(\btheta)\preccurlyeq M
\bfI_p$ for every $\btheta\in\mathbb R^p$. Choose
$\gamma$ and $h$ so that $\gamma \geqslant 5M$ 
and  $\sqrt{\kappa}\,\gamma h \leqslant 0.1$, 
where $\kappa = M/m$. Assume that $\bvartheta_0$ 
is independent of $\bv_0$ and that $\bv_0\sim
\mathcal N_p(0, \gamma\bfI_p)$. Then, for any
$n \geqslant 1$, the distribution $\nu_n^{\textup{ 
\sf KLMC}}$ of $\bvartheta_n^{\textup{\sf KLMC}}$
satisfies
\begin{align}
    \wass_2(\nu_n^{\textup{\sf KLMC}},\pi) &\leqslant 
    2\varrho^n  \wass_2(\nu_0,\pi)
    + 0.05\sqrt{\varrho^n
    \mathbb E[f(\bvartheta_0) - f(\btheta_*)]/m}
    + 0.9 \gamma h \sqrt{\kappa p/m} \,,
\end{align}
where $\varrho = \exp(-mh),$ and $
\btheta_*=\arg\min_{\btheta\in\RR^p} f(\btheta).$ 
\end{theorem}

Bounds on the error of KLMC under convexity 
assumption, or other related conditions, can be found 
in recent papers 
\citep{cheng2018underdamped,dalalyan_riou_2018,
monmarche2020highdimensional,MONMARCHE2023316}. 
Our result has the advantage of providing  
an upper bound with the best known dependence 
on the condition number $\kappa$ and having 
relatively small numerical constants, as shown
in the next corollary. 

\begin{corollary}
    Let $\varepsilon\in(0,0.1)$. 
    If $\gamma = 5M$, $\bvartheta_0 = \btheta_*$ and 
    we choose $h>0$ and $n\in \mathbb{N}$ so that
    \begin{align}
        \gamma h =  \varepsilon 
        {\kappa}^{-1/2},\quad
    \text{and} 
    \quad
        n\geqslant {5\kappa^{3/2}}\varepsilon^{-1}
        \log(20/\varepsilon)     
    \end{align}
    then we have $\wass_2(\nu_n^{\sf KLMC},\pi)\leqslant 
    \varepsilon\sqrt{p/m}$.  
\end{corollary}

It is worth noting that our error bounds, 
along with the other bounds mentioned 
previously under strong convexity, rely 
on the synchronous coupling between the 
KLMC and the KLD. However, in the case of 
the vanilla Langevin, it has been shown 
in \cite{durmus2019analysis} that the 
dependence of the error bound on $\kappa$ 
can be improved by considering other 
couplings (in their case, the coupling 
is hidden in the analytical arguments). 
We conjecture that the dependence on 
$\kappa$ in the kinetic Langevin Monte 
Carlo algorithm can also be improved 
through non-synchronous coupling. 
Specifically, we conjecture that the 
number of iterations required to achieve 
a $\wass_2$-error bounded by $\varepsilon 
\sqrt{p/m}$ should scale as $\kappa/
\varepsilon$ rather than $\kappa^{3/2}/
\varepsilon$, as obtained in previous 
work and in \Cref{thm:klmc1}.

\section{Discussion of assumptions 
and outlook}

The results presented in this paper provide easily computable guarantees for performing sampling with assured accuracy. These guarantees are conservative, implying that the actual sampling error may be smaller than $\varepsilon$ even if the upper bounds stated in our theorems are larger than $\varepsilon$. However, these bounds represent the most reliable technique available in the existing literature. 
The importance of having such guarantees is further emphasized by the lack of reliable practical measures to assess the quality of sampling methods. 
To better understand the computational complexity implied by our bounds for various Monte Carlo algorithms, we present in \Cref{tab:1} the number of gradient evaluations required to achieve the accuracy of $\varepsilon\sqrt{p/m}$ for different combinations of $(\varepsilon,\kappa)$.

\paragraph{Strong convexity}
The assumption of strong convexity is often 
seen as too restrictive. However, there are 
ways to relax this assumption. In our theorems, 
strong convexity is used for three purposes: 
(a) to ensure exponential convergence of the 
continuous-time dynamics, (b) to relate the 
potential's values to its gradient through 
the Polyak-Lojasiewicz condition 
$\|\nabla f(\btheta)\|^2\geqslant 2m (f 
(\btheta) - f(\btheta_*)) $ \citep{Polyak63,
Lojasiewicz}, and (c) to provide the following 
simple upper bound on the 2-Wasserstein distance 
$\wass_2( \delta_{\btheta_*},\pi) \leqslant 
\sqrt{p/m}$ \citep[Prop.~1]{durmus2019}. 
Instead of strong convexity, we can assume 
that the potential satisfies the $m$-PL 
condition and that the target distribution 
satisfies the Poincare inequality with a 
constant no larger than $1/m$. This leads 
to only minor changes in the proof techniques 
and results.

\begin{table}[t]
    \caption{The number of iterations that are 
    sufficient for the algorithms $\{$L,RL,KL,RKL$\}$MC to achieve an error in $\wass_2$
    distance bounded by $\varepsilon \sqrt{p/m}$, 
    provided that they are initialized at the minimum 
    of the potential $f$. }
    \centering\small
    \begin{tabular}{c||rrrrrr}
        \toprule
         $(\varepsilon,\kappa)$ 
         & $(0.1^1,10^1)$ & $(0.1^1,10^3)$ 
         & $(0.1^1,10^5)$ & $(0.1^1,10^7)$ 
         & $(0.1^1,10^9)$ & $(0.1^1,10^{11})$\\
         \midrule
         LMC & $1.2\times 10^4$ &  $1.2\times 
         10^6$ & $1.2\times 10^8$  & $1.2
         \times 10^{10}$ & $1.2 \times 10^{12}$ 
         & $1.2\times 10^{14}$\\
         RLMC & $3.6\times 10^3$ & $1.1\times 10^6$ 
         & $4.5\times 10^{8}$ & $2.0\times 10^{11}$ & 
         $9.3\times 10^{13}$ & $4.3\times 10^{16}$\\
         KLMC  & $8.4\times 10^3$ & $8.4\times 10^6$ 
         & $8.4\times 10^9$ & $8.4\times 10^{12}$  & 
         $8.4\times 10^{15}$ & $8.4\times 10^{18}$\\
         RKLMC & $1.0\times 10^4$ & $1.1\times 10^6$ & 
         $1.1\times 10^8$ & $1.3\times 10^{10}$  & 
         $2.2\times 10^{12}$ & $4.2\times 10^{14}$\\
         \toprule
         $(\varepsilon,\kappa)$ 
         & $(0.1^3,10^1)$ & $(0.1^3,10^3)$ 
         & $(0.1^3,10^5)$ & $(0.1^3,10^7)$ 
         & $(0.1^3,10^9)$ & $(0.1^3,10^{11})$\\
         \midrule
         LMC & $2.2\times 10^8$ & $2.2\times 10^{10}$ & 
         $2.2\times 10^{12}$ & $2.2\times 10^{14}$ & 
         $2.2\times 10^{16}$ & $2.2\times 10^{18}$ \\
         RLMC & $3.8\times 10^5$ & $6.8\times 10^{7}$ & 
         $2.0\times 10^{10}$ & $8.4\times 10^{12}$ & 
         $3.8\times 10^{15}$ & $1.7\times 10^{18}$ \\
         KLMC & $1.6\times 10^6$ & $1.6\times 
         10^{9\hphantom{0}}$ & $1.6\times 10^{12}$ & 
         $1.6\times 10^{15}$& $1.6\times 10^{18}$ & 
         $1.6\times 10^{21}$\\
         RKLMC & $4.5\times10^{5}$ & $4.5\times10^{7
         \hphantom{0}}$ & $4.5\times10^{\hphantom{0}9}$ 
         & $4.5 \times10^{11}$  & $4.7\times10^{13}$ 
         & $5.7\times10^{15}$\\
         \toprule
         $(\varepsilon,\kappa)$ 
         & $(0.1^5,10^1)$ & $(0.1^5,10^3)$ 
         & $(0.1^5,10^5)$ & $(0.1^5,10^7)$ 
         & $(0.1^5,10^9)$ & $(0.1^5,10^{11})$\\
         \midrule
         LMC & $3.2\times 10^{12}$ & $3.2\times 10^{14}$ 
         & $3.2\times 10^{16}$ & $3.2\times 10^{18}$  
         & $3.2\times 10^{20}$ & $3.2\times 10^{22}$\\
         RLMC & $4.6\times 10^{7\hphantom{0}}$ & $5.5
         \times 10^{9\hphantom{0}}$ & $9.9\times 10^{11}$ 
         & $3.0\times 10^{14}$  & $1.2\times 10^{17}$ & 
         $5.5\times 10^{19}$\\
         KLMC & $2.3\times 10^{8\hphantom{0}}$ & $2.3
         \times 10^{11}$ & $2.3\times 10^{14}$ & $2.3
         \times 10^{17}$  & $2.3\times 10^{20}$ & $2.3
         \times 10^{23}$\\
         RKLMC & $1.5\times 10^{7\hphantom{0}}$ & $1.5
         \times 10^{9\hphantom{0}}$ & $1.5\times 10^{11}$ 
         & $1.5\times 10^{13}$  & $1.5\times 10^{15}$ 
         & $1.5\times 10^{17}$\\
         \bottomrule
    \end{tabular}
    \label{tab:1}
\end{table}

Alternatively, we can assume that the function 
is only strongly convex outside a ball of 
radius $R>0$, whereas within the ball it is 
smooth but otherwise arbitrary. This approach requires an additional factor of order $e^{MR^2}$ in the number of iterations necessary to achieve a specified error level \citep{Cheng3,ma2018sampling}. We can also assume that the Markov semi-group has a spectral gap and use this gap in the risk bounds. However, this approach goes against the spirit of our paper, which aims to provide guarantees that are easy to interpret and verify.

Another important point to note is that the results obtained under the assumption of strong convexity can be used as ready-made results in other frameworks as well. For instance, this is applicable to weakly convex potentials or potentials supported on a compact set \citep{dalalyan2019bounding,dwivedi2018log,brosse17a}.

\paragraph{Smoothness} 
Smoothness of $f$ is a critical assumption for the 
results obtained in this paper. However, in statistical 
applications, this assumption may not hold, such as 
when using a Laplace prior. In such cases, various 
approaches have been proposed, mainly involving 
gradient approximation techniques, as explored in 
the literature \citep{Durmus3, chatterji2019langevin}. 
Our results open the door for similar extensions of 
the randomized midpoint method for such scenarios.

It should also be stressed that if the potential is 
more than twice differentiable with a bounded tensor of 
higher-order derivatives, then it is possible to
design Monte Carlo algorithms that perform better 
than the LMC and the KLMC \citep{dalalyan2019user,
dalalyan_riou_2018,ma2019there}. The same is true if 
the function $f$ has some specific structure \citep{Mou2021}.


\paragraph{Functional inequalities}
Functional inequalities such as the Poincar\'e and
the log-Sobolev inequalities provide a convenient
framework for analyzing sampling methods derived
from continuous-time Markov processes. This line
of research was developed in a series of papers
\citep{chewi2020exponential,Vempala_Wibisono,
Chewi_Erdogdu}. We believe our results can also
be reformulated using these inequalities; however, 
we opted for sticking to log-concave setting in
order to keep conditions easy to check. Note that
even for simple distributions such as the posterior
of the logistic model, the constant of the Poincar\'e
inequality is not known. 

\paragraph{Other distances}
The Wasserstein-2 distance, utilized in this paper, 
serves as a natural metric for measuring the error 
in sampling due to its connection with optimal 
transport. However, it is worth noting that recent 
literature on gradient-based sampling has explored 
other metrics such as total variation distance, KL 
divergence, and $\chi^2$ divergence
\citep{ma2019there,Vempala_Wibisono,
durmus2019analysis,chewi2020exponential,balasubramanian2022towards,zhang2023improved}. An interesting 
direction for future research involves establishing 
error guarantees for the randomized midpoint method 
with respect to these alternative distances.


\begin{ack}
This work was partially supported by the grant Investissements d’Avenir 
(ANR-11-IDEX0003/Labex Ecodec/ANR-11-LABX-0047) and the center \href{https://www.hi-paris.fr/}{Hi! PARIS}.
\end{ack}


\bibliography{bibliography}

\newpage

\appendix

\tableofcontents

\addtocontents{toc}{\protect\setcounter{tocdepth}{3}}

\include{appendix_A}
\clearpage
\include{appendix_B}
\clearpage
\include{appendix_C}


\end{document}

%% file: appendix_A.tex
\section{The proof of the upper bound on the error of
RLMC}\label{app:proof-rlmc}


This section is devoted to the proof of the upper bound on the 
error of sampling, measured in $\wass_2$-distance, of the randomized 
mid-point method for the vanilla Langevin Langevin diffusion. 
Since no other sampling method is considered in this section, without 
any risk of confusion, we will use the notation $\bvartheta_k$ instead
of $\bvartheta^{\sf RLMC}_k$ to refer to the $k$th iterate of the
RLMC. We will also use the shorthand notation
\begin{align}
    f_k = f(\bvartheta_k),\qquad 
    \nabla f_{k}:=\grad f(\bvartheta_k),\qquad 
    \text{and}\qquad f_{k+U}:=\grad f(\bvartheta_{k+U}).
\end{align}

\subsection{Proof of \Cref{thm:rlmc}}

    Let $\bvartheta_0\sim\nu_0$ and $\bL_0\sim\pi$ be two random vectors 
    in $\mathbb R^p$ defined on the same probability space. At this stage, the
    joint distribution of these vectors is arbitrary; we will take an infimum
    over all possible joint distributions with given marginals at the end of the proof. 
    Note right away that the condition 
    $Mh + \sqrt{\kappa} (Mh)^{3/2}\leqslant 1/4$ 
    implies that $Mh + (Mh)^{3/2}\leqslant 1/4$,
    which also yields $Mh\leqslant 0.18$. 

    Assume that on the same probability space, we can define a Brownian motion
    $\bW$, independent of $(\bvartheta_0,\bL_0)$, and an infinite sequence of iid random variables, uniformly distributed in $[0,1]$, $U_0,U_1,\ldots$, independent of $(\bvartheta_0,\bL_0,\bW)$. We define the Langevin diffusion 
    \begin{align}\label{eq-int-lang}
    		\bL_t = \bL_0 - \int_{0}^{t}\nabla f(\bL_s)\rmd s + \sqrt{2}\, \bW_t.
    \end{align}
        We also set
    \begin{align} 
        \bvartheta_{k+U} &=  \bvartheta_k  - hU_k\nabla f_k + \sqrt{2} \,\big(\bW_{(k+U_k)h} - \bW_{kh}\big) \label{meth-mid1a}\\
        \bvartheta_{k+1} & = \bvartheta_k  - h\nabla f_{k+U}  
    		 + \sqrt{2}\, (\bW_{(k+1)h} - \bW_{kh}).\label{meth-mid2a}
    \end{align} 

    One can check that this sequence $\{\bvartheta_k\}$ has exactly the
    same distribution as the sequence defined in \eqref{meth-mid1} and \eqref{meth-mid2}. Therefore,
    \begin{align}
        \wass_2^2(\nu_{k+1},\pi) \leqslant \mathbb E[\|\bvartheta_{k+1}- \bL_{(k+1)h}\|_2^2] := \|\bvartheta_{k+1}- \bL_{(k+1)h} 
        \|_{\Ltwo}^2: = x_{k+1}^2. 
    \end{align}
    We will also consider the Langevin process on the time interval $[0,h]$ given by 
    \begin{equation}\label{eq-int-lang1}
		\bL'_t = \bL'_0 - \int_{0}^{t}\nabla f(\bL'_s)\,\rmd s + \sqrt{2} \,(\bW_{kh+t} - \bW_{kh}),\qquad \bL'_0 = \bvartheta_k. 
    \end{equation}
    Note that the Brownian motion is the same as in \eqref{eq-int-lang}. 

    Let us introduce one additional notation, the average of 
    $\bvartheta_{k+1}$ with respect to $U_k$,
    \begin{align}
        \bar\bvartheta_{k+1} = \mathbb E[\bvartheta_{k+1} | 
        \bvartheta_k,\bW,\bL_0]. 
    \end{align}
    Since $\bL_{(k+1)h}$ is independent of $U_k$, it is clear that 
    \begin{align}
        x_{k+1}^2 
        & = \|\bvartheta_{k+1} - \bar\bvartheta_{k+1} 
        \|_{\Ltwo}^2 + \|\bar\bvartheta_{k+1} - \bL_{(k+1)h} 
        \|_{\Ltwo}^2.
    \end{align}
    Furthermore, the triangle inequality yields
    \begin{align}
        \|\bar\bvartheta_{k+1} - \bL_{(k+1)h} \|_{\Ltwo}
        &\leqslant \|\bar\bvartheta_{k+1} - \bL'_{h} \|_{\Ltwo} +
        \|\bL'_h - \bL_{(k+1)h} \|_{\Ltwo}.
    \end{align}
    Using the standard results on the convergence of 
    Langevin diffusions, we get
    \begin{align}
        \|\bL'_h-\bL_{(k+1)h}\|_{\Ltwo} \leqslant e^{-mh}\|\bL'_0-\bL_{kh}\|_{\Ltwo} = e^{-mh}\|\bvartheta_k - \bL_{kh}
        \|_{\Ltwo} = e^{-mh}x_k.
    \end{align}
    Therefore, we get
    \begin{align}
        x_{k+1}^2 
        & \leqslant \|\bvartheta_{k+1} - \bar\bvartheta_{k+1} 
        \|_{\Ltwo}^2 + \big(\| \bar\bvartheta_{k+1} -\bL'_h
        \|_{\Ltwo} + e^{-2mh}x_k \big)^2\\
        &= \big(e^{-mh}x_k  + \|\bar\bvartheta_{k+1} -\bL'_h
        \|_{\Ltwo} \big)^2 + \|\bvartheta_{k+1} - 
        \bar\bvartheta_{k+1} \|_{\Ltwo}^2 .\label{eq:4.4}
    \end{align}
    The last term of the right-hand side can be bounded as 
    follows
    \begin{align}
        \|\bvartheta_{k+1} - \bar\bvartheta_{k+1} \|_{\Ltwo}
        & = h\|\nabla f_{k+U} -\mathbb E_U[\nabla f_{k+U}]
        \|_{\Ltwo}\\
        &\leqslant h\|\nabla f_{k+U} - \nabla f 
        (\bvartheta_k)\|_{\Ltwo}. 
    \end{align}
    Using the definition of $\bvartheta_{k+U}$, we get
    \begin{align}
        \|\bvartheta_{k+1} - \bar\bvartheta_{k+1} \|_{\Ltwo}^2 
        &\leqslant (Mh)^2 \Big((1/3)h^2\|\nabla f_k
        \|^2_{\Ltwo} + hp\Big).\label{eq:4.5}
    \end{align}
    We will also need the following lemma, the proof of 
    which is postponed.
    
    
    \begin{lemma}\label{lem:A2}
        If $Mh\leqslant 0.18$, then 
        $\|\bar\bvartheta_{k+1}-\bL'_h\|_{\Ltwo} 
        \leqslant (Mh)^2  \big\{ 0.7 h
        \|\nabla f_k\|_{\Ltwo} + 1.2 \sqrt{h p}
        \big\}$.
    \end{lemma}

    One can check by induction that if for some $A\in [0,1]$ 
    and for two positive sequences $\{B_k\}$ and $\{C_k\}$
    the inequality $x_{k+1}^2\leqslant \big\{(1-A)x_k + C_k
    \big\}^2 + B_k^2$ holds for every integer $k\geqslant 0$,
    then\footnote{This is an extension of 
    \citep[Lemma 7]{dalalyan2019user}. It essentially relies 
    on the elementary $\sqrt{(a+b)^2 + c^2}\leqslant a+ 
    \sqrt{b^2 + c^2}$, which should be used to prove the 
    induction step.} 
    \begin{align}\label{eq:xk}
        x_n\leqslant (1 - A)^n x_0 + \sum_{k=0}^n (1-A)^{n-k} 
        C_k +\bigg\{\sum_{k=0}^n (1-A)^{2(n-k)}B_k^2\bigg\}^{1/2}
    \end{align}
    In view of \eqref{eq:xk}, \eqref{eq:4.4}, \eqref{eq:4.5} and
    \Cref{lem:A2}, for $\rho = e^{-mh}$, we get
    \begin{align}
        x_n\leqslant \rho^{n} x_0 &+ (Mh)^2\sum_{k=0}^n 
        \rho^{n-k} \big(0.7 h\|\nabla f_k 
        \|_{\Ltwo} + 1.2 \sqrt{hp}\big)\\
        & +  
        Mh\bigg\{\sum_{k=0}^n \rho^{2(n-k)} \big((1/3)h^2 
        \|\nabla f_k\|_{\Ltwo}^2 + hp\big)
        \bigg\}^{1/2}\\
        \leqslant \rho^{n} x_0 &+ 0.7 (Mh)^2 h\sum_{k=0}^n 
        \rho^{n-k} \|\nabla f_k \|_{\Ltwo} + 1.32
        \frac{ M^2 h\sqrt{hp}}{m} \\
        & + \frac{Mh^2}{\sqrt{3}}
        \bigg\{\sum_{k=0}^n \rho^{2(n-k)} \|\nabla f 
        (\bvartheta_k)\|_{\Ltwo}^2\bigg\}^{1/2} + 
        0.92 Mh\sqrt{p/m}.\label{eq:xn-bound}
    \end{align}

    We need a last lemma for finding a suitable upper
    bound on the right-hand side of the last display.
    \begin{lemma}\label{lem:A3}
    If $Mh\leqslant 0.18$ and $k\geqslant 1$, then the following inequalities hold
    \begin{align}
        h^2\sum_{k=0}^n \rho^{n-k} \|\nabla f ( 
        \bvartheta_k)\|^2_{\Ltwo} 
        &\leqslant 1.7 Mh\rho^n\|\bvartheta_0\|_{
        \mathbb L_2}^2 + 4.4 Mh (p/m) \leqslant 0.31
        \rho^n\|\bvartheta_0\|_{\Ltwo}^2+ 0.8 (p/m).
    \end{align}    
    \end{lemma}
    The claim of this lemma and together with \eqref{eq:xn-bound} 
    entail that
    \begin{align}
        x_n&\leqslant \rho^{n} x_0 + 0.7 (Mh)^2 h 
        \sum_{k=0}^n \rho^{n-k} \|\nabla f_k 
        \|_{\Ltwo} + \frac{Mh^2}{\sqrt{3}}
        \bigg\{\sum_{k=0}^n \rho^{2(n-k)} \|\nabla f
        (\bvartheta_k)\|_{\Ltwo}^2\bigg\}^{1/2}\\ 
        &\qquad\qquad + (1.32 \sqrt{\kappa Mh}  + 0.92)
        Mh\sqrt{p/m}\\
        &\leqslant \rho^{n} x_0 + \big( 0.74 \sqrt{
        \kappa M h} + 0.58 \big)Mh \bigg\{h^2\sum_{k=0}^n
        \rho^{n-k} \|\nabla f (\bvartheta_k)\|_{\mathbb
        L_2}^2\bigg\}^{1/2}\\ 
        &\qquad\qquad + \big(1.32 \sqrt{\kappa Mh}  
        + 0.92\big) Mh\sqrt{p/m}\\
        &\leqslant \rho^{n} x_0 + (0.42\sqrt{\kappa Mh} 
        + 0.33)Mh\rho^{n/2} \|
        \bvartheta_0 \|_{\Ltwo} + 
        \big(1.98 \sqrt{\kappa Mh}  + 1.44\big) Mh\sqrt{p/m}.
    \end{align}
    Assuming that $h$ is such that $(\sqrt{\kappa Mh} + 1)Mh
    \leqslant 1/4$ and noting that $\|\bvartheta\|_{\Ltwo} \leqslant \wass_2(\nu_0,\pi) + \sqrt{p/m}$, we 
    arrive at the desired inequality.

\subsection{Proof of technical lemmas}

In this section, we present the proofs of two technical 
lemmas that have been used in the proof of the main 
theorem. The first lemma provides an upper bound on the
error of the averaged iterate $\bar\bvartheta_{k+1}$ and
the continuous time diffusion $\bL'$ that starts from
$\bvartheta_k$ and runs until the time $h$. This upper
bound involves the norm of the gradient of the potential $f$
evaluated at $\bvartheta_k$. The second lemma aims at
bounding the discounted sums of the squared norms of 
these gradients. 

\subsubsection{Proof of \Cref{lem:A2} (one-step mean discretisation error)}
    We have
    \begin{align}
        \|\bar\bvartheta_{k+1}-\bL'_h\|_{\Ltwo} & =
        \bigg\|\bvartheta_{k}-h\mathbb E_U[\nabla f_{k+U}]
        -\bL'_0 +\int_0^h \nabla f(\bL'_s)\,\rmd s \bigg\|_{\Ltwo}\\
        & = \Big\|h\mathbb E_U[\nabla f_{k+U} - 
        \nabla f(\bL'_{Uh})] \Big\|_{\Ltwo}
        \leqslant Mh\big\|\bvartheta_{k+U} - \bL'_{Uh}\big\|_{\Ltwo}\\
        & \leqslant Mh\int_0^{h} \big\|\nabla f(\bL'_s) 
        - \nabla f(\bL'_0) \big\|_{\Ltwo}\,\rmd s.\label{eq:4.7}
    \end{align}    
    Let us define $\varphi(t) = \|\nabla f(\bL'_t)
    - \nabla f(\bL'_0) \|_{\Ltwo}$.  Using the Lipschitz
    continuity of $\nabla f$ and the definition of $\bL'$, 
    we arrive at
    \begin{align}
        \varphi(t)^2 & \leqslant M^2\bigg\{\bigg\|\int_0^t 
        \nabla f(\bL'_s)\,\rmd s\bigg\|_{\Ltwo}^2  + 2tp\bigg\}\\
        &\leqslant M^2\bigg\{\bigg(t\|\nabla f_k
        \|_{\Ltwo} + \int_0^t \big\|\nabla f(\bL'_s) - \nabla f(\bL'_0)\big\|_{\Ltwo} \,\rmd s \bigg)^2 + 2tp\bigg\}\\
        &\leqslant M^2\bigg\{ \int_0^t \big\|\nabla 
        f(\bL'_s) - \nabla f(\bL'_0)\big\|_{\Ltwo} \,\rmd s +  
        \sqrt{t^2\|\nabla f_k \|_{\Ltwo}^2 + 
        2tp}\bigg\}^2
    \end{align}
    or, equivalently,
    \begin{align}
        \varphi(t)
        &\leqslant M \int_0^t \varphi(s) \,\rmd s + M\sqrt{t^2 
        \|\nabla f_k \|_{\Ltwo}^2 + 2tp}. 
    \end{align}
    Using the Gr\"onwall inequality, we get
    \begin{align}\label{eq:phit}
        \varphi(t) \leqslant Me^{Mt} \sqrt{t^2 
        \|\nabla f_k \|_{\Ltwo}^2 + 2tp}.
    \end{align}
    Combining this inequality with the bound obtained in 
    \eqref{eq:4.7}, and using the inequality $e^{Mh}
    \leqslant 1.2$, we arrive at
    \begin{align}
        \|\bar\bvartheta_{k+1}-\bL'_h\|_{\Ltwo} & 
        \leqslant 1.2 M^2h \int_0^h \sqrt{s^2\|\nabla 
        f_k\|_{\Ltwo}^2 + 2sp}\,\rmd s\\
        & \leqslant 1.2 M^2h \sqrt{h}\bigg\{\int_0^h 
        \big(s^2 \|\nabla f_k\|_{\Ltwo}^2 
        + 2sp\big)\,\rmd s\bigg\}^{1/2}\\
        &\leqslant 1.2 M^2h \sqrt{h}\big\{ (h^3/3)
        \|\nabla f_k\|_{\Ltwo}^2 + h^2p
        \big\}^{1/2}.
    \end{align}
    This completes the proof.

\subsubsection{Proof of \Cref{lem:A3} (Discounted sum of squared gradients)}
    We have 
    \begin{align}
        f_{k+1} &\leqslant f_k + 
        \nabla f_k^\top (\bvartheta_{k+1} -
        \bvartheta_k) + \frac{M}{2}\|\bvartheta_{k+1} -
        \bvartheta_k\|_2^2\\
        &\leqslant f_k  -h 
        \nabla f_k^\top\nabla f_{k+U} +
        \sqrt{2}\nabla f_k^\top\bxi_h + \frac{M}{2}\|h\nabla f_{k+U} - \sqrt{2}\,\bxi_k\|_2^2\\
        &\leqslant f_k  -h 
        \|\nabla f_k\|_2^2 + 
        Mh\|\nabla f_k\|_2\|\bvartheta_{k+U}-\bvartheta_k\|_2 + 
        \sqrt{2}\nabla f_k^\top\bxi_h + \frac{M}{2}\|h\nabla f_{k+U} - \sqrt{2}\,\bxi_k\|_2^2
        \label{eq:n1}.
    \end{align}
    One checks that
    \begin{align}
        \|\bvartheta_{k+U}-\bvartheta_k\|_{\Ltwo}^2 & = 
        h^2\|U\nabla f_k\|_{\Ltwo}^2 + 
        2hp \mathbb E[U] = (h^2/3)\|\nabla f_k\|_{\Ltwo}^2 + hp \leqslant 0.011\frac{\|\nabla f_k\|_{\Ltwo}^2}{M^2} + hp 
    \end{align}
    and, therefore,
    \begin{align}
        M\|\nabla f_k\|_2\|\bvartheta_{k+U}-\bvartheta_k\|_2 &\leqslant 
        \big(0.011\|\nabla f_k\|_{\Ltwo}^4 + M^2hp \|\nabla f_k\|_2^2\big)^{1/2}\\
        &\leqslant 0.105 \|\nabla f_k\|_{\Ltwo}^2 + 
        4.55 M^2hp\\
        &\leqslant 0.105 \|\nabla f_k\|_{\Ltwo}^2 + 
        0.82 Mp.
    \end{align}
    Furthermore,
    \begin{align}
        \|h\nabla f_{k+U} - \sqrt{2}\,\bxi_k\|_{\Ltwo} 
        &\leqslant \|h\nabla f_k - \sqrt{2}\,\bxi_k\|_{\Ltwo} 
        + h \|\nabla f_{k+U} - \nabla f_k\|_{\Ltwo} \\
        & \leqslant \sqrt{h^2\|\nabla f_k\|_{\Ltwo}^2 + hp} + Mh
        \|\bvartheta_{k+U} - \bvartheta_k\|_{\Ltwo}\\
        &\leqslant \sqrt{h^2\|\nabla f_k\|_{\Ltwo}^2 + hp} + h\sqrt{0.011\|\nabla f_k\|_{\Ltwo}^2 + M^2hp}\\
        &\leqslant \sqrt{h^2\|\nabla f_k\|_{\Ltwo}^2 + hp} + \sqrt{0.011h^2\|\nabla f_k\|_{\Ltwo}^2 + 0.18^2hp}
    \end{align}
    implying that
    \begin{align}
        \frac{M}{2}\|h\nabla f_{k+U} - \sqrt{2}\,\bxi_k\|_{\Ltwo}^2 
        &\leqslant \frac{M}{2}\big({1.37h^2\|\nabla f_k\|_{\Ltwo}^2 + 1.4hp}\big)\\
        &\leqslant 0.124 h^2\|\nabla f_k\|_{\Ltwo}^2 + 
        0.7 Mhp.
    \end{align}
    Combining these inequalities with 
    \eqref{eq:n1}, we get
    \begin{align}\label{eq:n2}
        \mathbb E[f_{k+1}] \leqslant 
        \mathbb E[f_k] - 0.771h\|\nabla f_k\|_{\Ltwo}^2 
        +1.52Mhp.
    \end{align}
    Set $S_n(f) = \sum_{k=0}^n \rho^{n-k} f_k$ and 
    $S_n(\nabla f^2) = \sum_{k=0}^n \rho^{n-k} \|
    \nabla f_k\|_{\Ltwo}^2$. Using \Cref{lem:rho-sum}, we get
    \begin{align}
        \mathbb E[f_{n+1}] - \rho^{n}\mathbb E[f_0] 
        + \rho S_n(f) \leqslant S_n(f) - 0.771 h
        S_n(\nabla f^2) + \frac{1.52 Mhp}{1-\rho}
    \end{align}
    Since $mh\geqslant 1-\rho\geqslant 0.915 mh$, we get
    \begin{align}
        0.771 h
        S_n(\nabla f^2) &\leqslant \rho^{n}\mathbb E[f_0] 
        + (1-\rho)  S_n(f)  + {1.67 \kappa p}\\
        &\leqslant \rho^{n}\mathbb E[f_0] 
        + mh  S_n(f)  + {1.67 \kappa p}\\
        &\leqslant \rho^{n}\mathbb E[f_0] 
        + 0.5 h  S_n(\nabla f^2)  + {1.67 \kappa p}
    \end{align}
    where the last line follows from the Polyak-Lojasiewicz inequality. Rearranging the terms, we get 
    \begin{align}\label{eq:n3}
        hS_n(\nabla f^2) &\leqslant 
        3.7\rho^{n}\mathbb E[f_0] 
        +  {6.2 \kappa p}
    \end{align}
    Note that \eqref{eq:n3} is obtained under the 
    Polyak-Lojasiewicz condition, without explicitly 
    using the strong convexity of $f$. However, using
    the latter property, we can obtain a similar
    inequality with slightly better constants.
    
    Indeed, \eqref{eq:n2} yields
    \begin{align}\label{hgradf}
        h\mathbb E[\|\nabla f_k\|_2^2] \leqslant 1.3 \big(\mathbb E[f_k] - 
        \mathbb E[f_{k+1}]\big) + 1.98 Mhp.
    \end{align}
        In what follows, without loss of generality, we
    assume that $f(\btheta^*) = \min_{\btheta} f(\btheta) = 0$. 
    In view of \eqref{hgradf}, we have
    \begin{align}
        h\sum_{j=0}^k \rho^{k-j} \|\nabla f(\bvartheta_j)
        \|^2_{\Ltwo} 
        &\leqslant 
        1.3 \sum_{j=0}^k \rho^{k-j}(\mathbb E[f(\bvartheta_j) - f(\bvartheta_{j+1})]) 
        + \frac{1.98 Mhp}{1-e^{-mh}}\\
        &\leqslant 
        1.3 \big(\rho^{k}\mathbb E[f(\bvartheta_0)] - 
         \mathbb E[f_{k+1}]\big)
        +
        1.3  \sum_{j=1}^k \rho^{k-j}(1-\rho)\mathbb E[f(\bvartheta_j)] 
        + 2.1 \kappa p\\
        &\leqslant 
        1.3  \rho^{k+1}\mathbb E[f(\bvartheta_0)] +
        1.3 (1-\rho) \sum_{j=0}^k \rho^{k-j}\mathbb E[f(\bvartheta_j)] 
        + 2.1 \kappa p\\
        &\leqslant 
        \frac{1.3 M }{2}\rho^{k+1}\|\bvartheta_0\|_{\Ltwo}^2 +
        \frac{1.3}2 M(1-\rho) \sum_{j=0}^k \rho^{k-j}\|\bvartheta_j\|_{\Ltwo}^2
        + 2.1 \kappa p.
    \end{align}
    We have, in addition 
    \begin{align}
        \|\bvartheta_{k+1}\|_{\mathbb L_2}^2 &= \|\bvartheta_k 
        -h\nabla f_k\|^2_{\Ltwo} + 2hp
        \leqslant (1-mh)^2\|\bvartheta_k\|_{\Ltwo}^2
        + 2hp.
    \end{align}
    Therefore,
    \begin{align}
        \|\bvartheta_k\|^2_{\Ltwo} &\leqslant 
        (1-mh)^{2k} \|\bvartheta_0\|_{\Ltwo}^2
        + \frac{2hp}{2mh - (mh)^2} \leqslant (1-mh)^{2k}
        \|\bvartheta_0\|_{\Ltwo}^2 + \frac{1.1p}{m}.
    \end{align}
    Using this inequality in conjunction with the fact that $1-mh\leqslant \rho$, we arrive at
    \begin{align}
    h\sum_{j=0}^k \rho^{k-j} \|\nabla f(\bvartheta_j)\|^2_{\Ltwo} 
        &\leqslant 
        \frac{1.3 M}{2} \rho^{k+1} \|\bvartheta_0 
        \|_{\Ltwo}^2 + \frac{1.3}{2} M\rho^{2k} \|
        \bvartheta_0\|_{\Ltwo}^2 + 2.9 \kappa p.
    \end{align}
    This completes the proof of the lemma.

%% file: appendix_B.tex
\section{The proof of the upper bound on the 
error of KLMC }\label{app:proof-klmc}

The goal of this section is to present the proof 
of the bound on the error of sampling of the 
``standard'' discretization of the kinetic Langevin 
diffusion. With a slight abuse of language, we will 
call it Euler-Maruyama discretized kinetic Langevin 
diffusion, or kinetic Langevin Monte Carlo (KLMC). 
To avoid complicated notation, and since there is 
no risk of confusion, throughout this section
$\bvartheta_k$ and $\bv_k$ will refer to 
$\bvartheta_k^{\sf KLMC}$ and $\bv_k^{\sf KLMC}$, 
respectively. We will also use the following 
shorthand notation:
\begin{align}
    f_n = f(\bvartheta_n),\quad \bg_n = 
    \nabla f_n = \nabla f(\bvartheta_n),\quad 
    \eta = \gamma h, \quad M_\gamma = M/\gamma.
\end{align}
The advantage of dealing with $\eta$ instead
of  $h$ is that the former is scale-free.

Note that the iterates of KLMC  satisfy
\begin{align}
    \bv_{n+1} & = (1- \alpha \eta ) \bv_n  - 
    \alpha \eta\, \bg_n + \sqrt{2\gamma\eta} \,
    \sigma \bxi_n \label{vn1}\\
    \bvartheta_{n+1} & = \bvartheta_n + \gamma^{-1}
    \eta \big(\alpha \bv_n - 
    \beta \eta \bg_n + \sqrt{2\gamma\eta}
    \, \tilde\sigma \bar\bxi_n\big),\label{eq:klmc}
\end{align}
where 
\begin{align}
    \alpha &= \frac{1-e^{-\eta}}{\eta} \in 
    (0,1),\qquad \beta = \frac{e^{-\eta} - 1 + 
    \eta  }{\eta^2} \in (0,1/2),\\
    \sigma^2 &=  \frac{1-e^{-2\eta}}{2\eta} 
    \in (0,1),\qquad \tilde\sigma^2 =  \frac{2(1 - 
    2\eta + 2\eta^2 -e^{-2\eta} )}{
    (2\eta)^3} \in (0,1/3)
\end{align}
and $\bxi_n,\bar\bxi_n$ are two $\mathcal N_p
(\mathbf 0, \mathbf I_p)$-distributed random 
vectors independent of $(\bvartheta_n,\bv_n)$.     

Since we assume throughout this section that 
$2Mh\leqslant 0.1$,  $\gamma\geqslant 2M$ and 
$\kappa \geqslant 10$, we have 
\begin{align}\label{eq:mh/gamma}
    \alpha  = \frac{1 - \exp(-\eta)}{\eta} 
    \geq 0.95,\quad\text{and}
    \quad 
    mh = \frac{Mh}{\kappa} \leqslant \frac{Mh}{10}
    \leqslant\frac{1}{200}.
\end{align} 
The latter, in particular, implies the following 
bound for $\varrho$:
\begin{equation}\label{eq:rhohgamma}
     1 - mh \leqslant \varrho = e^{- mh} \leqslant 
     1 - {0.99mh} = 1 -  {0.99m \eta}/{\gamma}.
\end{equation}

For any sequence $\omega = (\omega_n)_{n\in \NN}$ of 
real numbers, we denote by $S_n(\omega)$ the 
$\rho$-discounted sum $\sum _{k=0}^n \rho^{n-k} 
\omega_k$. Below we present a simple lemma for the 
function $S_n(\cdot)$ that we will use repeatedly 
in this proof.

\begin{lemma}[Summation by parts]\label{lem:rho-sum}
    Suppose $\omega = (\omega_n)_{n\in \NN}$ is a
    sequence of real numbers and define $S_n^{+1} 
    (\omega) := \sum_{k = 0}^{n} \varrho^{n-k} 
    \omega_{k+1}$. Then, the following identity is true
    \begin{equation}
        S_n^{+1} (\omega) = \omega_{n+1} - \varrho^{n+1} 
        \omega_0 + \varrho S_n(\omega).
    \end{equation}
\end{lemma}
\begin{proof}
    The proof is based on simple algebra:\\
    $\displaystyle~\qquad\qquad
    S_n^{+1} (\omega) 
    = \omega_{n+1} + \sum_{j = 1}^{n} \varrho^{n-j+1}
    \omega_{j} = \omega_{n+1} + \varrho\brr{S_n(\omega) 
    - \varrho^n \omega_0}.
    $
\end{proof}

\subsection{Exponential mixing of continuous-time
kinetic Langevin diffusion}

Consider the kinetic Langevin diffusions
\begin{align}
    d\bL_t &=\bV_t\, dt\qquad
    d\bV_t = -\gamma \bV_t\, dt-\gamma \grad f(\bL_t)dt 
    + \sqrt{2}\gamma\, d\bW_t
    \label{eq:klmc_new}
\end{align}

\begin{proposition}
    \label{prop:klmc_contr}
    Let $\bV_0,\bL_0$ and $\bL_0'$ be random vectors 
    in $\RR^p$.  Let $(\bV_t,\bL_t)$ and $(\bV_t', \bL_t')$ 
    be kinetic Langevin diffusions defined 
    in~\eqref{eq:klmc_new} driven by the same Brownian motion 
    and starting from $(\bV_0,\bL_0)$ and $(\bV_0',\bL'_0)$ 
    respectively. It holds for any $t\geqslant 0$ that
    \begin{align}
        \bigg\|\bfC \begin{bmatrix}
            \bV_t-\bV'_t \\
            \bL_t-\bL'_t
        \end{bmatrix}
        \bigg\|
        \leqslant e^{-\{m\wedge(\gamma - M)\}t} 
        \bigg\|\bfC \begin{bmatrix}
            \bV_0-\bV'_0 \\
            \bL_0-\bL'_0
        \end{bmatrix}
        \bigg\|,\quad \text{with}\quad 
        \bfC = \begin{bmatrix}
            \bfI_p & \mathbf 0_p\\
            \bfI_p & \gamma\bfI_p
        \end{bmatrix}.
    \end{align}
    \end{proposition}
\begin{proof}[Proof of~\Cref{prop:klmc_contr}]
Set $\bY_t:=\bV_t-\bV'_t+\gamma(\bL_t-\bL'_t)$, 
$\bZ_t:=\bV_t - \bV'_t$, that is 
\begin{align}
    \begin{bmatrix}
        \bZ_t\\
        \bY_t
    \end{bmatrix} = \bfC
    \begin{bmatrix}
        \bV_t - \bV_t'\\
        \bL_t - \bL_t'
    \end{bmatrix}.
\end{align}
We note that by the Taylor expansion, we have
\begin{align}
    \grad f(\bL_t)-\grad f(\bL'_t)=\bfH_t(\bL_t-\bL'_t)\,,
\end{align}
where $\bfH_t:=\int_0^1\grad^2 f(\bL_t-x(\bL_t-\bL'_t))\,\rmd x$. 
By the definition of $(\bV_t,\bL_t)$ and $(\bV'_t,\bL'_t)$, we find
\begin{align}
    \frac{\rmd}{\rmd t} (\bV_t - \bV'_t + \gamma(\bL_t-\bL'_t))
    & = -\gamma\bfH_t(\bL_t-\bL'_t)\\
    & = -\bfH_t(\bY_t - \bZ_t)\,.
\end{align}
Similarly, we obtain
\begin{align}
    \frac{\rmd}{\rmd t} (\bV_t-\bV'_t)
    &= -\gamma(\bV_t-\bV'_t)
    -\gamma\bfH_t(\bL_t-\bL'_t)\\
    &=-\gamma \bZ_t-\bfH_t(\bY_t - \bZ_t)\,.
\end{align}
This implies
\begin{align}
 \frac{d}{dt}\big[\norm{\bY_t}^2+\norm{\bZ_t}^2\big]
    & = 2\bY_t^\top(-\bfH_t\bY_t+\bfH_t\bZ_t)
        +2\bZ_t^\top(-\gamma \bZ_t-\bfH_t\bY_t+\bfH_t\bZ_t)\\
    &\leqslant 2\big( -m\norm{\bY_t}^2-\gamma \norm{\bZ_t^2} 
    + M\norm{\bZ_t}^2\big)\\
    &\leqslant -2(m\wedge  (\gamma - M))\bigg\|
 \begin{bmatrix}
 \bZ_t\\
 \bY_t
 \end{bmatrix}
 \bigg\|^2\,.
\end{align}
Invoking Gronwall's inequality, we get
\begin{align}
    \bigg\|
     \begin{bmatrix}
     \bZ_t\\
     \bY_t
     \end{bmatrix}
     \bigg\|
    \leqslant \exp\big( -\{m\wedge  (\gamma - M)\}t\big)
    \bigg\|
     \begin{bmatrix}
        \bZ_0\\
        \bY_0
     \end{bmatrix}
     \bigg\|
\end{align}
as desired.
\end{proof}

\subsection{Proof of \Cref{thm:klmc1}}

Let $\bvartheta_n$, $\bv_n$ be the iterates of 
the KLMC algorithm. Let $(\bL_t,\bV_t)$ be the 
kinetic Langevin diffusion,  coupled with
$(\bvartheta_n, \bv_n)$ through the same Brownian
motion $(\bW_t; t\geqslant 0)$ and starting from 
a random point $(\bL_0, \bV_0)\propto \exp(- 
f(\by) + \frac{1}{2\gamma} \|\bv \|_2^2)$ such that
$\bV_0=\bv_0$. This means that 
\begin{align}
    \bv_{n+1} & = \bv_n e^{-\eta}  - \gamma 
    \int_0^{h} e^{-\gamma(h-s)}\,\rmd s 
    \nabla f(\bvartheta_n) + \sqrt{2}\,\gamma 
    \int_0^h e^{-\gamma(h-s)}\,\rmd\bW_s\\
    \bvartheta_{n+1} & = \bvartheta_n + \int_0^h 
    \bigg(\bv_n e^{-\gamma u}  -  \gamma\int_0^u 
    e^{-\gamma(u-s)}\,\rmd s \nabla f(\bvartheta_n) 
    + \sqrt{2}\,\gamma \int_0^u e^{-\gamma (u-s)}\,
    \rmd\bW_s\bigg)\,\rmd u\,.
\end{align}
We also consider the kinetic Langevin diffusion,
$(\bL',\bV')$, defined on $[0,h]$ with the starting 
point $(\bvartheta_n, \bv_n)$ and driven by the 
Brownian motion $(\bW_{nh+t} - \bW_{nh};t \in[0,h])$. 
It satisfies
\begin{align}
    \bV'_t & = \bv_n e^{-\gamma t}  -  \gamma \int_0^t 
        e^{-\gamma(t-s)} \nabla f(\bL'_s)\,\rmd s +
        \sqrt{2}\,\gamma  \int_0^t e^{-\gamma(t-s)} 
        \,\rmd\bW_s\\
    \bL'_t & = \bvartheta_n +  \int_0^t \bV'_s\,
    \rmd s\,.
\end{align}
Our goal will be to bound the term $x_n$ defined by
\begin{align}\label{eq:xn}
    x_n = \bigg\| \bfC
    \begin{bmatrix}
        \bv_n- \bV_{nh}\\
        \bvartheta_n-\bL_{nh}
    \end{bmatrix}
    \bigg\|_{\Ltwo}
    \quad \text{with}\quad 
    \bfC = \begin{bmatrix}
    \mathbf I_p & \mathbf{0}_p\\
    \mathbf I_p & \gamma\mathbf I_p
    \end{bmatrix}.
\end{align}
The triangle inequality yields
\begin{align}
    x_{n+1} &\leqslant \bigg\| \bfC
    \begin{bmatrix}
        \bv_n- \bV'_{h}\\
        \bvartheta_n-\bL'_{h}
    \end{bmatrix}
    \bigg\|_{\Ltwo}
    + \bigg\| \bfC
    \begin{bmatrix}
        \bV'_h- \bV_{nh}\\
        \bL'_h-\bL_{nh}
    \end{bmatrix}
    \bigg\|_{\Ltwo} \\
    &\leqslant 
    \bigg\| \bfC
    \begin{bmatrix}
        \bv_n- \bV'_{h}\\
        \bvartheta_n-\bL'_{h}
    \end{bmatrix}
    \bigg\|_{\Ltwo}
    +
    \varrho x_{n}\\
    &\leqslant \varrho x_n  + \sqrt{2}\,
    \|\bv_{n+1} - \bV'_h\|_{\Ltwo}  + \gamma\|
    \bvartheta_{n+1} - \bL'_h \|_{\Ltwo},
    \label{eq:iterKLMC}
\end{align}
where the second inequality follows from 
\Cref{prop:klmc_contr} (see also the
proof of \citep[Prop.~1]{dalalyan_riou_2018}), 
while the third inequality is a consequence of the 
elementary inequality $\sqrt{a^2 + (a+b)^2}\leqslant 
\sqrt{2}\,a + b$ for $a,b\geqslant 0$.

The next lemma gives an upper bound on the terms
appearing in the right-hand side of 
\eqref{eq:iterKLMC}. 

\begin{lemma}
\label{lem:klmc}
If~ $\nabla f$ is $M$-Lipschitz continuous, then for 
every step-size $\eta = \gamma h\geqslant 0$ and 
every $\gamma\geqslant 0$, the following holds
\begin{align}
\|\bv_{n+1}-\bV'_h\|_{\Ltwo}
    & \leqslant \tfrac16\big\{2\sqrt{\gamma p  \eta} + 
    3\|\bv_n\|_{\Ltwo} + \eta  \|\bg_n\|_{\Ltwo} 
    \big\} M_\gamma \eta^2 e^{M_\gamma \eta^2/2}\\
\gamma\|\bvartheta_{n+1}-\bL'_h\|_{\Ltwo}
    &\leqslant \tfrac16 \big( 0.6\sqrt{\gamma p\eta} 
    + \|\bv_n\|_{\Ltwo} + 0.25\eta\|\bg_n\|_{\Ltwo} +  
    \big)M_\gamma \eta^3 e^{M_\gamma\eta^2/2},
    \label{eq:thetadiff}
\end{align}
where $M_\gamma = M/\gamma$.
\end{lemma}

For $\eta\leqslant 0.2$ and $\gamma\geqslant 2M$, 
\Cref{lem:klmc} implies
\begin{align}
\|\bv_{n+1}-\bV'_h\|_{\Ltwo}
& \leqslant  M_\gamma \eta^2 (0.15\sqrt{\gamma p}
+ 0.51 \|\bv_n\|_{\Ltwo}
+0.17 \eta\|\bg_n\|_{\Ltwo}) \\
\gamma\|\bvartheta_{n+1}-\bL'_{h}\|_{\Ltwo}
& \leqslant  M_\gamma \eta^3 (0.046\sqrt{\gamma p}
+ 0.17\|\bv_n\|_{\Ltwo}
+0.043 \eta\|\bg_n\|_{\Ltwo}).
\end{align}
Therefore,
\begin{align}
    \sqrt{2}\, \| \bv_{n+1}  - \bV'_h\|_{\Ltwo}
    + \gamma\| \bvartheta_{n+1} - \bL'_h\|_{\Ltwo} 
    \leqslant M_\gamma\eta^2 \big(0.23\sqrt{\gamma p} 
    +  0.74\|\bv_n\|_{\Ltwo} + 0.25\eta\|\bg_n
    \|_{\Ltwo}).
    \label{eq:rem1}
\end{align}
Combining \eqref{eq:iterKLMC} and \eqref{eq:rem1}, 
we get
\begin{align}
    x_{n+1} &\leqslant \varrho x_n + M_\gamma\eta^2 
    \big(0.23\sqrt{\gamma p } +  0.74\|\bv_n\|_{\Ltwo} 
    + 0.25\eta\|\bg_n\|_{\Ltwo}).
\end{align}
From the last display, we infer that
\begin{align}\label{eq:klmc_ctr}
    x_{n} &\leqslant \varrho^{n} x_{0}  + M_\gamma
    \eta^2 \sum_{k=0}^{n-1} \varrho^{n-1-k}
    \big(0.23\sqrt{\gamma p} +  0.74\|\bv_k\|_{\Ltwo} 
    + 0.25\eta\|\bg_k\|_{\Ltwo}).
\end{align}
This implies that
\begin{align}\label{eq:klmc_ctr1}
    x_n 
    &\leqslant \varrho^n x_0 + \frac{0.23 
    M_\gamma\eta^2\sqrt{\gamma p}}{1-\varrho} + 
    0.74M_\gamma\eta^2 \sum_{k=1}^n \varrho^{n-k}
    \big(\| \bv_{k-1} \|_{\Ltwo} + 0.33\eta\| 
    \bg_{k-1}\|_{\Ltwo} \big)\\
    &\leqslant \varrho^n x_0 + \frac{0.23 M_\gamma
    \eta^2\sqrt{\gamma p }}{1-\varrho} + \frac{0.74
    M_\gamma\eta^2}{\sqrt{1-\varrho}} \bigg\{
    \sum_{k=1}^n \varrho^{n-k}\big(\| \bv_{k-1} 
    \|_{\Ltwo}^2 + 0.33\eta^2\| \bg_{k-1} \|_{\Ltwo}^2 
    \big)\bigg\}^{1/2}.
\end{align}
In view of \eqref{eq:rhohgamma}, 
$\varrho \leqslant 1 - 0.99 m\eta/\gamma$ and 
\begin{align}\label{eq:klmc_ctr2}
    x_n  &\leqslant \varrho^n x_0 + 0.233 \kappa \eta 
    \sqrt{\gamma p} + {0.74M_\gamma\eta}\bigg\{\frac{
    \gamma\eta}{m\varrho} \sum_{k=0}^n \varrho^{n-k} 
    \big(\| \bv_{k} \|_{\Ltwo}^2 + 0.33 \eta^2\| 
    \bg_{k} \|_{\Ltwo}^2\big)\bigg\}^{1/2}.
\end{align}

\begin{proposition}
 \label{lem:sum2} 
Assume that $\bv_0\sim\mathcal N(0,\gamma \bfI_p)$ is 
independent of $\bvartheta_0$. If $\gamma\geqslant 5M$, 
$\kappa \geqslant 10$ and $\eta \leqslant 1/10$ then
\begin{align}
    \eta\sum_{k=0}^{n}  \varrho^{n - k} \|\bg_k\|_{\Ltwo}^2
    &\leqslant 4.42\varrho^n \gamma 
    \,\mathbb E[ f_0 ] + \frac{1.11 \gamma^2 p}{m} + 
    4.98\big(x_n + 0.96\sqrt{\gamma p}\big)^2\\
    \eta\sum_{k=0}^{n}  \varrho^{n - k} \|\bv_k\|_{\Ltwo}^2
    &\leqslant 3.93\varrho^n\gamma  \mathbb E[f_0]
    + \frac{1.87\gamma^2 p}{m}
     + 3.2 \big(x_n + 0.96\sqrt{\gamma p}\big)^2.
\end{align}
\end{proposition}

We can apply \Cref{lem:sum2} and $\varrho \geqslant
0.998$ to infer that 
\begin{align}\label{eq:klmc_ctr2}
    x_n 
    &\leqslant \varrho^n x_0  + 0.233 \kappa \eta 
    \sqrt{\gamma p} + 0.74 M_\gamma \eta \bigg\{ 
    \frac{\gamma}{m\varrho}\Big(3.98\varrho^n 
    \gamma\,\mathbb E[ f_0 ] + \frac{1.87\gamma^2 
    p}{m} + 3.25\big(x_n + 0.96\sqrt{\gamma p}\big)^2\Big) 
    \bigg\}^{1/2}\\
    &\leqslant \varrho^n x_0 + 0.233 \kappa \eta 
    \sqrt{\gamma p} + 0.74 M_\gamma \eta \bigg\{ 
    \frac{\gamma}{m}\Big(3.99\varrho^n \gamma\,
    \mathbb E [f_0] + \frac{1.88 \gamma^2 p}{m} + 
    3.26\big(x_n + 0.96\sqrt{\gamma p}\big)^2\Big) 
    \bigg\}^{1/2}\\
    &\leqslant \varrho^n x_0  + + 0.233 \kappa \eta 
    \sqrt{\gamma p} + \frac{1.54 M \eta}{\sqrt{m}}
    \sqrt{\varrho^n\mathbb E[f_0]} + 0.62\sqrt{\kappa}
    \,\eta x_n + 0.74 M_\gamma \eta\Big(\frac{1.88
    \gamma^3 p}{m^2} + \frac{3\gamma^2 p}{m}\Big)^{1/2}\\
    &\leqslant \varrho^n x_0  + 0.62\sqrt{\kappa}
    \,\eta x_n + \frac{1.54 M \eta}{\sqrt{m}} 
    \sqrt{\varrho^n\mathbb E[f_0]} + \frac{ 
    M\eta \sqrt{\gamma p} }{m}\big(0.233 + 0.74\sqrt{1.94} \Big).
\end{align}
Therefore, under the condition $\sqrt{\kappa}\,\eta
\leqslant 0.1$,
\begin{align}
    x_n &\leqslant 1.07\varrho^n x_0 + \frac{1.65 M 
    \eta}{\sqrt{m}}\sqrt{\varrho^n\mathbb E[f_0]} + 
    \frac{1.35 M\eta \sqrt{\gamma p} }{m}.
\end{align}
Finally, one can check that $2x_n^2\geqslant \gamma^2 
\|\bvartheta_n - \bar\bvartheta_n\|_{\Ltwo}^2 \geqslant 
\gamma^2 \wass_2^2(\nu_n^{\textup{\sf KLMC}},\pi)$ and 
$x_0 = \gamma \|\bvartheta_0 - \bL_0\|_{\Ltwo}= 
\gamma\wass_2 (\nu_0,\pi)$. This completes the proof of 
the theorem. 

\subsection{Proof of \Cref{lem:sum2} (discounted sums 
of squared gradients and velocities)}

To ease the notation, we set  $z_n := \E[\bv_n^{\top} 
\bg_n]$ and define
\begin{align}
    S_n(z) &:= \sum_{k=0}^{n}\varrho^{n-k} z_k,
    &S_n(g^2)&:= \sum_{k=0}^n\varrho^{n-k} 
    \norm{\bg_k}^2_{\Ltwo}, \\    
    S_n(f) &:= \sum_{k=0}^{n}\varrho^{n-k}
        \mathbb E[f_{k}],
    &S_n(v^2)& := \sum_{k=0}^{n}\varrho^{n-k}
        \|\bv_{k}\|_{\Ltwo}^2\,.
\end{align}

Throughout the proof, we will need some technical 
results that will be stated as lemmas and their proof 
will be postponed to \Cref{sec:B-tech-lemmas}.
\begin{lemma}
\label{lem:inprod}
    If for some $M\geqslant 0$, the gradient $\nabla f$ 
    is $M$-Lipschitz continuous, then for every 
    step-size $h>0$ and every $\gamma>0$ it holds for 
    the KLMC iterates defined in~\eqref{eq:klmc} that 
    \begin{align}
    \big|z_{n+1} - (1-\alpha \eta) z_n + \alpha \eta
    \|\bg_n\|_{\Ltwo}^2\big|
    \leqslant  
    \eta M_\gamma\big(\|\bv_n\|_{\Ltwo}^2 &+ \tfrac58
    \eta^2\|\bg_n\|^2_{\Ltwo} + \tfrac43\eta \gamma p 
     - \tilde\alpha \eta  z_n\big)
    \end{align}
    for some positive number $\tilde\alpha\eta
    \leqslant 0.14$.
\end{lemma}
Since $M_\gamma\leqslant 1/5$ and $\eta\leqslant 0.1$, 
we have
\begin{align}
    \alpha -\tfrac58 M_\gamma \eta^2 \geqslant 
    \tfrac1\eta(1 - e^{-\eta}) -\tfrac18 \eta^2
    \geqslant 10(1- e^{-0.1}) - \tfrac18 0.1^2 
    \geqslant  0.94.
\end{align}
Therefore, we can rewrite the claim of \Cref{lem:inprod} 
with the notation $\tilde\beta = \eta M_\gamma 
\tilde\alpha$ as follows:
\begin{align}
    z_{n+1}
    \leqslant (1-\alpha \eta - \tilde \beta\eta) 
    z_n + 0.2\eta\|\bv_n\|_{\Ltwo}^2 + 0.67\gamma 
    p\eta^2  -  0.94 \eta \| \bg_n\|^2_{\Ltwo}.
\end{align}
\begin{lemma}\label{lem:klmc3}
    Let $\tilde\beta\leqslant 0.014$ and $\eta
    \in[0,0.1]$. If $z_0 = 0$ and the sequences 
    $\{z_n\}\subset \mathbb R$, $\{\bv_n\}\subset 
    \mathbb R^p$ and $\{\bg_n\}\subset\mathbb R^p$ 
    satisfy the inequality 
    \begin{align}
        z_{n+1}
        \leqslant (e^{-\eta} - \tilde\beta\eta) 
        z_n + 0.2\eta\|\bv_n\|_{\Ltwo}^2 + 0.67
        \gamma p\eta^2  -  0.94 \eta \| \bg_n
        \|^2_{\Ltwo}\label{eq:Zn}
    \end{align}
    then for every $\varrho\in [0,1]$ such that 
    $\varrho \geqslant e^{-\eta}$, it holds that
    \begin{align}
        S_n( g^2)  \leqslant 1.09\alpha (S_n(z))_- 
        + 0.213 S_n(v^2) + \frac{0.73 \eta\gamma 
        p}{1-\varrho} - \frac{1.07 z_{n+1}}{\eta},
        \label{eq:sngrad-snz}
    \end{align}
    where $(S_n(z))_- = \max(0,S_n(z))$ is the 
    negative part of $S_n(z)$ and $\alpha = (
    1-e^{-\eta})/\eta$. 
\end{lemma}

In order to get rid of the last term in 
\eqref{eq:sngrad-snz}, we need a bound on 
$(S_n(z))_-$.  To this end, we use the smoothness of 
the function $f$, in conjunction with \eqref{eq:klmc}, 
to infer that 
\begin{align}
    2\gamma\mathbb E[f_{n+1} -f_n] 
    &\leqslant 2\gamma \mathbb E[\bg_n^\top(
    \bvartheta_{n+1} - \bvartheta_n)]+ M_\gamma \|
    \gamma(\bvartheta_{n+1} - \bvartheta_n)\|_{
    \Ltwo}^2 \\
    &\leqslant 2\alpha  \eta (1 -  \beta M_\gamma  
    \eta^2) z_n - {\beta \eta^2}(2 -\beta M_\gamma 
    \eta^2) \| \bg_n\|^2_{\Ltwo} + {M_\gamma\alpha^2 
    \eta^2} \|\bv_n\|^2_{\Ltwo} + \tfrac{2}{3}{ 
    M_\gamma \eta^3 \gamma p}\\
     &\leqslant 2\alpha_0\eta z_n  - 0.96\eta^2 \| 
     \bg_n\|^2_{\Ltwo} + 0.0182\eta \|\bv_n\|^2_{\Ltwo} 
     + \tfrac{2}{15}{\eta^3 \gamma p}
\end{align}
with $\alpha_0 = \alpha(1-\beta M_\gamma\eta^2) 
\geqslant \alpha (1- 0.1\times 0.1^2) \geqslant 0.999
\alpha$. 
\begin{lemma}\label{lem:klmc4}
    Let $\alpha_0,\gamma,\eta > 0$. If the sequences 
    $\{F_n\}\subset \mathbb 
    R$, $\{\bg_n \}\subset \mathbb R^p$ and $\{\bv_n\} 
    \subset\mathbb R^p$ satisfy $F_n\geqslant 0$ and 
    \begin{align}\label{eq:Fn2}
    2(F_{n+1} -F_n) \leqslant 2\alpha_0\eta z_n
     + 0.0182\eta \|
     \bv_n\|^2_{\Ltwo} + \tfrac{2}{15}{\eta^3 \gamma p}
    \end{align}
    then, for every $\varrho \in (0,1)$, it holds that
    \begin{align}
        \alpha_0 (\eta S_n(z))_-
        \leqslant \varrho^n F_0 + (1-\varrho)
        \gamma S_n(F) + 0.0182\eta S_n(v^2) +  \frac{2
        \eta^3 \gamma p}{15(1 - \varrho)}.
    \end{align}
\end{lemma}
In view of the strong convexity of the potential 
function and the assumption that $f(\btheta_*) = 0$, 
the Polyak-Lojasiewicz inequality
\begin{equation}
    f_n \leqslant \tfrac{1}{2m} \norm{\bg_n}^2
\end{equation}
holds true. This implies that $(1-\varrho)S_n(f) 
\leqslant (\eta/\gamma) m S_n(f) \leqslant 
\frac12(\eta/\gamma) S_n(g^2)$. Combining this
inequality with the claim of \Cref{lem:klmc4}, 
applied to $F_n = \gamma \mathbb E[f_n]$, we get
\begin{align}\label{eq:-snz}
    \boxed{0.999\alpha (S_n(z))_-
    \leqslant \frac{\varrho^n \gamma}{\eta} 
    \mathbb E[f_0] + 0.5 S_n(g^2) + 0.0182 
    S_n(v^2) +  \frac{0.14\eta\gamma^2 p}{m}.}
\end{align}
Let us now combine \eqref{eq:sngrad-snz} and  
\eqref{eq:-snz}:
\begin{align}
    S_n( g^2) &\leqslant 
    \frac{1.1\varrho^n \gamma}{\eta} \mathbb E[f_0] 
    + 0.55 S_n(g^2) + 0.02 S_n(v^2) +  
    \frac{0.16\eta\gamma^2 p}{m}\\
    &\qquad+ 0.213 S_n(v^2) +
    \frac{0.73 \gamma^2 p}{m} + \frac{1.07|z_{n+1}|
    }{\eta} \\    
    &\leqslant \frac{1.1\varrho^n \gamma}{\eta} 
    \mathbb E[f_0] + 0.55 S_n(g^2) + 0.223 
    S_n(v^2) +  \frac{0.75\gamma^2 p}{m} + 
    \frac{1.07|z_{n+1}|}{\eta}.
\end{align}
Subtracting $0.55 S_n(g^2)$ from both sides and
dividing by $0.45$, we obtain 
\begin{equation}\label{eq:sng-snv}
    \boxed{S_n( g^2) \leqslant 
    \frac{2.45\varrho^n\gamma \mathbb E[f_0]}{
    \eta} + 0.5 S_n(v^2)  + \frac{1.7\gamma^2 p}{m}
    +\frac{2.38|z_{n+1}|}{\eta}.
 }
\end{equation}
Let us now derive a bound for $S_n(v^2)$. 
We start with the following property, which is 
a direct consequence of the definition of 
$\bv_{n+1}$:
\begin{align}
    \|\bv_{n+1}\|_{\Ltwo}^2 - \|\bv_n\|_{\Ltwo}^2 
    &\leqslant -\alpha\eta (2-\alpha \eta) \|\bv_n
    \|^2_{\Ltwo} -2\alpha \eta(1-\alpha \eta) z_n  
    + \alpha^2 \eta^2\|\bg_n\|_{\Ltwo}^2 + 2\eta 
    \gamma p.
\end{align}
Using the same technique as before and applying 
\Cref{lem:rho-sum}, we deduce the following:
\begin{align}
   (\varrho -1) S_n( v^2) - \varrho^n \norm{\bv_0}_{\Ltwo}^2
   & = (\varrho -1) S_n( v^2) - \varrho^n \gamma p\\
   &\leqslant -\alpha\eta (2-\alpha \eta)S_n(v^2) 
    -2\alpha \eta(1-\alpha \eta) S_n(z) 
    + \alpha^2 \eta^2S_n(g^2) + \frac{2\eta\gamma p
    }{1-\varrho}.
\end{align}
Therefore, since $\varrho\geqslant 1-\tfrac{m}{\gamma}
\eta$, 
\begin{align}
    (2\alpha -\alpha^2\eta  - \tfrac{m}{\gamma})\eta 
    S_n( v^2) \leqslant  - 2\alpha \eta(1-\alpha \eta) 
    S_n(z) + (\alpha \eta)^2 S_n(g^2) + \frac{2.021
    \gamma^2 p}{m}.
\end{align}
Since $\alpha = (1-e^{-\eta})/\eta$ with $\eta\leqslant 
0.1$, from the last display, we infer that 
\begin{align}
    S_n( v^2) 
    &\leqslant 1.02 \alpha\big(S_n(z )\big)_-
    + 0.51\eta S_n(g^2) + \frac{1.13\gamma^2 p}{m\eta}.
\end{align}
Combining this inequality with \eqref{eq:-snz}, implies
\begin{align}
    S_n( v^2) 
    &\leqslant 
    \frac{1.03 \varrho^n\gamma }{\eta} \mathbb [f_0]
        + 0.562 S_n (g^2)
        + 0.019 S_n(v^2) 
        + \frac{0.2\eta\gamma^2 p}{ m}
    + 0.51\eta S_n(g^2) + \frac{1.13\gamma^2 p}{m\eta}\\
    &\leqslant \frac{1.03 \varrho^n\gamma }{\eta} f_0 
        + 0.62 S_n (g^2) + 0.019 S_n(v^2) 
        +\frac{1.13\gamma^2 p}{m\eta}
\end{align}
Therefore, subtracting $ 0.019 S_n(v^2)$ and dividing by
$(1-0.019)$, we get
\begin{align}
    \boxed{S_n( v^2) 
    \leqslant 
    \frac{1.1 \varrho^n \gamma }{\eta} f_0 
        + 0.64 S_n(g^2) + \frac{1.16\gamma^2 p}{m\eta}}
        \label{eq:snv2}
\end{align}
Combining \eqref{eq:sng-snv} and \eqref{eq:snv2}, we 
arrive at
\begin{align}
    S_n(v^2) &\leqslant 
    \frac{1.1 \varrho^n \gamma}{\eta}\mathbb E[f_0]
    + \frac{1.16\gamma^2 p}{m\eta} + 0.64  \Big(\frac{2.45
    \varrho^n \gamma}{\eta}\mathbb E[f_0] +  0.5 S_n(v^2)  
    + \frac{1.7\gamma^2 p}{m} + \frac{2.38|z_{n+1}|}{\eta}
    \Big)\\
    &\leqslant \frac{ 2.67\varrho^n \gamma}{\eta}  
    f_0 + \frac{1.27\gamma^2 p}{m\eta} + \frac{1.53 
    |z_{n+1}| }{\eta}  + 0.32 S_n(v^2).
\end{align}
Therefore, subtracting $0.32S_n(v^2)$ and dividing by 
$(1-0.32)$, we get
\begin{align}\label{eq:snv22}
    \boxed{S_n( v^2) 
    \leqslant 
    \frac{ 3.93\varrho^n\gamma }{\eta}  f_0 
    + \frac{1.87\gamma^2 p}{m\eta}
     + \frac{2.25 |z_{n+1}| }{\eta}.}
\end{align}
Once again, combining with \eqref{eq:sng-snv}, we get
\begin{align}
    S_n( g^2) &\leqslant \frac{2.45\varrho^n
    \gamma }{\eta}\mathbb E[f_0] + 0.5 \Big( \frac{ 
    3.93\varrho^n \gamma}{\eta} f_0 + \frac{1.87
    \gamma^2 p}{m\eta} + \frac{2.25 |z_{n+1}| }{\eta}
    \Big)  + \frac{1.7 \gamma^2 p}{m} + \frac{2.38 
    |z_{n+1}|}{\eta}
\end{align}
that leads to
\begin{align}
    \boxed{
    S_n(g^2) \leqslant \frac{4.42\varrho^n \gamma}{
    \eta}\,\mathbb E[f_0] + \frac{1.11 \gamma^2 p}{
    m\eta} + \frac{3.51|z_{n+1}|}{\eta} .}
    \label{eq:Sng22}
\end{align}
The last lemma we need is the one providing an upper
bound on $|z_{n+1}|$. 
\begin{lemma}\label{lem:zn}
    For every $\eta \leqslant 0.1$ and $\gamma
    \geqslant {5M}$, we have
    \begin{align}
        |z_{n+1}| \leqslant 
        \big(1.19 x_n + 1.14\sqrt{\gamma p}\big)^2,
    \end{align}
    where $x_n$ is given by \eqref{eq:xn}.
\end{lemma}
Using \Cref{lem:zn} in conjunction with \eqref{eq:snv2}
and \eqref{eq:Sng22}, we arrive at the inequalities 
stated in \Cref{lem:sum2}.

\subsection{Proof of Lemma~\ref{lem:klmc} 
(one-step discretization error)}

We use the notation
\begin{align}
    \psi_0(t) = e^{-\gamma t},\qquad
    \psi_1(t) = \frac{1 - e^{-\gamma t}}{\gamma},\qquad   
    \psi_2(t) = \frac{e^{-\gamma t} -1 + \gamma t}{\gamma}
\end{align}
and note that
\begin{align}
    \psi_1(t)\leqslant t,\qquad \psi_2(t) \leqslant 0.5
    \gamma t^2. 
\end{align}
Furthermore, 
 \begin{align}
\|\bv_{n+1} - \bV'_h\|_{\Ltwo} 
    &= \gamma \bigg\|\int_0^h e^{-\gamma(h-s)} \big(
    \nabla f (\bL'_s) - \nabla f(\bvartheta_n)\big)\,
    \rmd s \bigg\|_{\Ltwo}\\
    &\leqslant  \gamma\int_0^h e^{-\gamma (h-s)}\big\| 
    \nabla f(\bL'_s) - \nabla f(\bvartheta_n)\big\|_{
    \Ltwo}\,\rmd s\\
    &\leqslant  M\gamma\int_0^h \big\|\bL'_s - 
    \bvartheta_n \big\|_{\Ltwo}\,\rmd s,
\label{lem21:1}
\end{align}
where the last implication is due to the $M$-smoothness 
of the potential function $f$. On the one hand, for every 
$s\in[0,h]$, we have
\begin{align}
\bL'_s - \bvartheta_n 
    &=\int_0^s \bV'_u \,\rmd u\\
    &=\psi_1(s)\,\bv_n - \gamma
    \int_0^s\int_0^u e^{-\gamma(u-t)} 
    \big(\nabla f(\bL'_t)-\nabla f(\bvartheta_n)\big)\,
    \rmd t\rmd u - \psi_2(s)\, \bg_n\\ 
    &\qquad + \sqrt{2}\,\gamma \int_0^s\int_0^u e^{- 
    \gamma (u-t)} \,\rmd\bW_t\,\rmd u\\
    &=\psi_1(s)\,\bv_n  - \psi_2(s)\, \bg_n + \sqrt{2}
    \,\gamma\int_0^s\psi_1(t)\,\rmd\bW_t\\
    &\qquad - \gamma
    \int_0^s \psi_1(s-t) 
    \big(\nabla f(\bL'_t)-\nabla f(\bvartheta_n)\big)\,
    \rmd t. 
\label{lem21:3}
\end{align}
Therefore,
\begin{align}
    \big\|\bL'_s - \bvartheta_n\big\|_{\Ltwo} 
    &\leqslant s\,\|\bv_n\|_{\Ltwo} + 0.5\gamma s^2
    \, \|\bg_n\|_{\Ltwo} + \sqrt{2p s/3}\,\gamma s 
    + M\gamma \int_0^s (s-t) \big\|\bL'_t -
    \bvartheta_n\big\|_{\Ltwo}\,\rmd t .
\label{lem21:4}
\end{align}
The last inequality combined with $s-t\leqslant 
h-t$ allows us to use the Gr\"onwall lemma, which 
implies that
\begin{align}
\big\|\bL'_s - \bvartheta_n\big\|_{\Ltwo} 
    &\leqslant  \big(s\,\|\bv_n\|_{\Ltwo} + 0.5
    \gamma s^2\, \|\bg_n\|_{\Ltwo} + \sqrt{(2/3)p s}
    \,\gamma s\big) e^{M\gamma s(h-0.5s)}\\
    &\leqslant  \big(s\|\bv_n\|_{\Ltwo}  + 0.5
    \gamma s^2\, \|\bg_n\|_{\Ltwo} + \sqrt{(2/3)p s}
    \,\gamma s\big) e^{0.5M\gamma h^2}.
\label{lem21:5}
\end{align}
Combining the last display with \eqref{lem21:1}, 
we get
\begin{align}
\big\|\bv_{n+1} - \bV'_h\big\|_{\Ltwo} 
    &\leqslant  \big\{\tfrac12 \|\bv_n\|_{\Ltwo} +
    \tfrac{1}{6}\eta  \|\bg_n\|_{\Ltwo} + 0.33 
    \sqrt{\gamma p  \eta}\big\} Mh^2\gamma e^{ 
    M\gamma h^2/2}\\
    &\leqslant  \big\{\tfrac12
    \|\bv_n\|_{\Ltwo}  +\tfrac{1}{6}\eta  \|\bg_n
    \|_{\Ltwo} + 0.33\sqrt{\gamma p  \eta}\big\} 
    M_\gamma \eta^2 e^{M_\gamma \eta^2/2}.
\label{lem21:6}
\end{align}
This completes the proof of the first inequality. 
To prove the second one, we again use the update 
rules of $\btheta_{n+1}$ and $\bL^\prime$:
\begin{align}
\big\|\bvartheta_{n+1} - \bL'_h\big\|_{\Ltwo} 
    &= \gamma\bigg\|\int_0^h \int_0^t e^{-\gamma^2
    (t-s)}\big(\nabla f(\bL'_s) - \nabla f
    (\bvartheta_n)\big)\,\rmd s \,\rmd t 
    \bigg\|_{\Ltwo}\\
    &\leqslant \gamma \int_0^h\int_0^t \big\|
    \nabla f (\bL'_s) - \nabla f(\bvartheta_n)
    \big\|_{\Ltwo}\,\rmd s \,\rmd t \\
    &\leqslant  M\gamma \int_0^h\int_0^t \big\|
    \bL'_s - \bvartheta_n\big\|_{\Ltwo}\,\rmd s 
    \,\rmd t.
\end{align}
The last term can be bounded using \eqref{lem21:5}. 
This yields
\begin{align}
\gamma\big\|\bvartheta_{n+1} - \bL'_h\big\|_{\Ltwo} 
    & \leqslant   M_\gamma \gamma^3 e^{M_\gamma
    \eta^2/2} \int_0^h\int_0^t \big(s\| \bv_n 
    \|_{\Ltwo} + 0.5\gamma s^2\|\bg_n\|_{\Ltwo}+ 
    \sqrt{(2/3)ps^3}\,\gamma \big)\,\rmd s\,\rmd t\\
    & \leqslant  M_\gamma \eta^3 e^{M_\gamma\eta^2/2} 
    \big(\tfrac16\|\bv_n\|_{\Ltwo} + \tfrac1{24}\eta
    \|\bg_n\|_{\Ltwo} + 0.1\sqrt{p\eta\gamma}  \big)\\
    &\leqslant (1/6) M_\gamma \eta^3 e^{M_\gamma 
    \eta^2/2} \big(\|\bv_n\|_{\Ltwo} + 0.25\eta\|
    \bg_n\|_{\Ltwo} + 0.6\sqrt{p\eta\gamma} \big)
\end{align}
as desired.

\subsection{Proofs of the technical lemmas used in 
\Cref{lem:sum2}}
\label{sec:B-tech-lemmas}

\subsubsection{Proof of \Cref{lem:inprod}}

Since $z_n = \mathbb E[\bg_n^\top \bv_n]$, we have
\begin{align}
    \big|z_{n+1} - 
    z_n  - \mathbb E[
    \bg_n^\top (\bv_{n+1} -\bv_n)]\big| = \big|
    \mathbb E\big[(\bg_{n+1} - \bg_n)^\top \bv_{n+1} 
    ]\big|.\label{eq:lem5-eq0}
\end{align}
On the one hand, definition \eqref{vn1} of 
$\bv_{n+1}$  yields 
\begin{equation}\label{eq:lemma5-eq1}
     \mathbb E[\bg_n^\top (\bv_{n+1} -\bv_n)]  =
    -\alpha\eta \mathbb E[\bg_n^\top \bv_n]-\alpha 
    \eta \|\bg_n\|_{\Ltwo}^2.    
\end{equation}
On the other hand, the Cauchy-Schwartz inequality 
implies
\begin{align}
    \big|\mathbb E\big[(\bg_{n+1} - \bg_n)^\top 
    \bv_{n+1} ]\big| &\leqslant \big\|\bg_{n+1} -
    \bg_n\big\|_{\Ltwo}\|\bv_{n+1}\|_{\Ltwo} \\
    & \leqslant M\|\bvartheta_{n+1} - \bvartheta_n 
    \|_{\Ltwo}\|\bv_{n+1}\|_{\Ltwo} . 
\end{align}
Similarly, using update rules \eqref{vn1} and 
\eqref{eq:klmc} of the KLMC, and the triangle 
inequality we get 
\begin{align}
    \gamma^2\|\bvartheta_{n+1} - \bvartheta_n
    \|^2_{\Ltwo}
    &\leqslant  \eta^2\|\bv_n\|_{\Ltwo}^2 + 
    \tfrac14\eta^4 \|\bg_n\|_{\Ltwo}^2 -2\alpha\beta
    \eta^2 z_n + \tfrac23\eta^3\gamma p\label{eq:lem5-eq2}\\
    \|\bv_{n+1}\|^2_{\Ltwo}&\leqslant 
    \|\bv_n\|_{\Ltwo}^2 + \eta^2 \|\bg_n\|_{\Ltwo}^2 
    - 2\alpha\eta(1-\alpha\eta) z_n + 
    2\eta \gamma p.\label{eq:lem5-eq3}
\end{align}
Hence,
\begin{align}
    \tfrac\gamma\eta\|\bvartheta_{n+1} - \bvartheta_n
    \|_{\Ltwo} \|\bv_{n+1}\|_{\Ltwo}
    &\leqslant  \frac{(\gamma/\eta)^2\|\bvartheta_{n+1} 
    - \bvartheta_n \|_{\Ltwo}^2 + 
    \|\bv_{n+1}\|_{\Ltwo}^2}{2} \\
    &\leqslant \|\bv_n\|_{\Ltwo}^2 + \tfrac58
    \eta^2 \|\bg_n\|_{\Ltwo}^2 - \alpha\eta(\beta + 
    1-\alpha\eta) z_n + \tfrac43\eta \gamma p.
    \label{eq:lemma5-eq3}
\end{align}
Therefore, combining \eqref{eq:lem5-eq0}, 
\eqref{eq:lemma5-eq1} and \eqref{eq:lemma5-eq3}, we get
\begin{align}
    \big|z_{n+1} - (1-\alpha \eta) z_n + \alpha \eta
    \|\bg_n\|_{\Ltwo}^2\big|
    \leqslant  
    \eta M_\gamma\big(\|\bv_n\|_{\Ltwo}^2 &+ \tfrac58
    \eta^2\|\bg_n\|^2_{\Ltwo} + \tfrac43\eta \gamma p 
    \big)\\
    & - \eta^2M_\gamma\underbrace{\alpha (\beta 
    + 1-\alpha \eta)}_{:=\tilde\alpha} z_n
\end{align}
with $\tilde\alpha\eta \leqslant \alpha\eta(1.5 - 
\alpha \eta)\leqslant 0.14$, as desired.

\subsubsection{Proof of \Cref{lem:klmc3}}
    We apply inequality \eqref{eq:Zn} for every index 
    $k \leqslant n$ multiply each side by 
    $\varrho^{n-k}$:
    \begin{align}
        \varrho^{n-k} z_{k+1}
        \leqslant (e^{-\eta} - \tilde\beta\eta)
        \varrho^{n-k}z_k +
        0.2  \varrho^{n-k}\eta\|\bv_k\|_{\Ltwo}^2 +
        0.67\eta^2\gamma p  \varrho^{n-k} - 0.94 
        \eta\varrho^{n-k} \| \bg_k\|^2_{\Ltwo}.
    \end{align}
    Summing over $k$, applying \Cref{lem:rho-sum} 
    and taking into account that $z_0 = 0$, we get 
    \begin{align}
        z_{n+1} + \varrho S_n(z) 
        & \leqslant (e^{-\eta} - \tilde\beta\eta) 
        S_n(z) + 0.2  \eta  S_n(v^2) + \frac{0.67
        \eta^2\gamma p  }{1-\varrho}
        - 0.94 \eta S_n(g^2)\\
        &\leqslant (e^{-\eta} - \tilde\beta\eta) 
        S_n(z) + 0.2  \eta  S_n(v^2) + \frac{0.68 
        \eta^2 \gamma p}{1-\varrho} - 0.94 \eta S_n(g^2).
    \end{align}
    This implies that
    \begin{align}
        0.94 \eta S_n( g^2)  
        &\leqslant (\varrho - e^{-\eta} + \tilde\beta
        \eta )(- S_n(z)) + 0.2  \eta  S_n(v^2) + 
        \frac{0.68 \eta^2 \gamma p}{1-\varrho} - z_{n+1}.
    \end{align}
    Note that $\rho - e^{-\eta} \geqslant 0$ and 
    $\varrho - e^{-\eta}  + \tilde\beta\eta
    \leqslant 1 - e^{-\eta} + 0.014\eta\leqslant 
    1.02(e^{-\eta} - 1) = 1.02 \alpha \eta$. 
    Therefore,
    \begin{align}
        0.94 \eta S_n( g^2)  
        &\leqslant 1.02\alpha \eta (S_n(z))_- + 0.2\eta  
        S_n(v^2) + \frac{0.68 \eta^2 \gamma p}{1-\varrho} 
        - z_{n+1}.
    \end{align}
    Dividing both sides of the last display by $0.94\eta$, 
    we get
    \begin{align}
        S_n( g^2)  
        \leqslant 1.09\alpha (S_n(z))_- + 0.213 S_n(v^2) 
        + \frac{0.73  \eta\gamma p}{1-\varrho} - 
        \frac{1.07 z_{n+1}}{\eta}.
    \end{align}
    This completes the proof of the lemma.

\subsubsection{Proof of \Cref{lem:klmc4}}

    We write inequality \eqref{eq:Fn2} for all indices 
    $k$ and multiply both sides of it by $\varrho^{n-k}$. 
    Summing the obtained inequalities and applying 
    \Cref{lem:rho-sum}, we obtain the following:
    \begin{equation}\label{eq:sn_f}
        \begin{aligned}
        2(\varrho - 1) S_n(F) - 2\varrho^n 
        F_0 \leqslant 2\alpha_0 \eta S_n(z)
        + 0.0.182\eta S_n(v^2) 
        + \frac{2\eta^3 \gamma p}{15(1 - \varrho)},
        \end{aligned}
    \end{equation}
    where the left-hand side is obtained using 
    \Cref{lem:rho-sum} and the fact that $F_n 
    \geqslant 0$. Rearranging the terms and dividing
    by 2, we obtain
    \begin{align}
        -\alpha_0 \eta S_n(z)
        \leqslant \varrho^n F_0 + (1-\varrho)
        S_n(F) + 0.0.182\eta S_n(v^2) +  \frac{2\eta^3 
        \gamma p}{15(1 - \varrho)}.
    \end{align}
    Since the right-hand side of the last display is 
    nonnegative, we infer that
    \begin{align}
        \alpha_0 (\eta S_n(z))_-
        \leqslant \varrho^n  F_0 + (1-\varrho)
        S_n(F) + 0.0.182\eta S_n(v^2) +  \frac{2\eta^3 
        \gamma p}{15(1 - \varrho)},
    \end{align}
    which coincides with the claim of the lemma.

\subsubsection{Proof of \Cref{lem:zn}}
    
    In view of \Cref{lem:inprod}, we have
    \begin{align}
        |z_{n+1}| 
        &\leqslant  (e^{-\eta}+0.3\eta^2) |z_n| + \eta 
        \|\bg_n\|_{\Ltwo}^2 + \eta \big(0.2 \|\bv_n 
        \|_{\Ltwo}^2 + 0.2\eta^2\| \bg_n\|^2_{\Ltwo} 
        + 0.027 \eta \gamma p\big)\\
        &\leqslant 0.56 (\|\bg_n\|_{\Ltwo} +\|\bv_n
        \|_{\Ltwo})^2  + 0.0027\gamma p\\
        &\leqslant 0.56\big(\|\bg_n - \nabla f(
        \bL_{nh})\|_{\Ltwo} + \|\bv_n - \bV_{nh}
        \|_{\Ltwo} + \sqrt{Mp} + \sqrt{\gamma p}\big)^2 
        + 0.0027 \gamma p,
    \end{align}
    where we have used the facts $\|\nabla f(\bL_{nh})
    \|_{\Ltwo} = \int \|\nabla f\|^2\,d\pi\leqslant Mp$
    \citep[Lemma 3]{dalalyan2019user} and $\mathbb E
    [\|\bV_{nh}\|^2] = \gamma p$. 
    Finally, one can note that
    \begin{align}
        \|\bg_n - \nabla f(\bL_{nh})\|_{\Ltwo}
        + \|\bv_n - \bV_{nh}\|_{\Ltwo} &\leqslant
        0.5\gamma \|\bvartheta_n - \bL_{nh}\|_{\Ltwo}
        + \|\bv_n - \bV_{nh}\|_{\Ltwo} \\
        &\leqslant
        0.5 \|\bv_n-\bV_{nh}+\gamma(\bvartheta_n - 
        \bL_{nh})\|_{\Ltwo} + 1.5 \|\bv_n - \bV_{nh} 
        \|_{\Ltwo} \\
        &\leqslant \sqrt{5/2}\, x_n.
    \end{align}
    Therefore, 
    \begin{align}
        |z_{n+1} |
        &\leqslant \big(
        1.19 x_n + 1.09\sqrt{\gamma p}
        \big)^2 + 0.0027 \gamma p
        \leqslant 
        \big(1.19 x_n + 1.14\sqrt{\gamma p}
        \big)^2.
    \end{align}
    This completes the proof of the lemma.

%% file: appendix_C.tex
\section{The proof of the upper bound on the error of 
RKLMC}\label{app:proof-rklmc}

Consider the underdamped Langevin diffusion
\begin{align}
\label{eq:uld}
\rmd \bL_t= \bV_t\,\rmd t,\qquad \text{where}
\qquad
\rmd\bV_t=- \gamma\bV_t\,\rmd t - \gamma\nabla f(\bL_t)\,\rmd t 
+\sqrt{2}\gamma\,\rmd \bW_t
\end{align}
for every $t\geqslant 0$, with  given initial
conditions $\bL_0$ and $\bV_0$. Throughout this section, 
we assume that $\bV_0\sim\mathcal N_p(0,\gamma\bfI_p)$ 
is independent of $\bL_0$, and the couple $(\bV_0,\bL_0)$
is independent of the Brownian motion $\bW$. We also
assume that $\bL_0$ is drawn from the target distribution
$\pi$; this implies that the process $(\bL_t,\bV_t)$ is
stationary.

In the sequel, we use the following shorthand notation 
\begin{align}
    \eta = \gamma h,\quad g = \nabla f ,\qquad 
    f_n = f(\bvartheta_n),\qquad \bg_n = g(\bvartheta_n),
    \qquad \bg_{n+U} = g(\bvartheta_{n+U}),\qquad M_\gamma 
    = M/\gamma.
\end{align}

The randomized midpoint discretization---proposed
and studied in \citep{shen2019randomized}---of the 
kinetic Langevin process~\eqref{eq:exactsol}, can be
written as 
\begin{align}
\bvartheta_{n+U}
&=\bvartheta_n+\frac{1-e^{-U\eta}}{\gamma}\bv_n
-\int_0^{U h}(1-e^{-\gamma (U h-s)})\,\rmd s\,\nabla f_n
+\sqrt2\,\int_0^{U h} 
(1-e^{-\gamma (U h -s)})\,d\bar{\bW}_s\\
\bvartheta_{n+1}
&=\bvartheta_n+\frac{1-e^{-\eta}}{\gamma}\bv_n
-  \eta\frac{1-e^{-\eta(1 - U)}}{\gamma}\nabla f_{n+U}
+\sqrt{2}\int_0^h (1-e^{-\gamma (h-s)})\,\rmd\bar{\bW}_s\\
\bv_{n+1}
&=\bv_n e^{-\eta}- \eta e^{-\gamma (h - U h)}
\nabla f_{n+U} + \sqrt{2}\,\gamma  
\int_0^h e^{-\gamma (h-s)}\,\rmd \bar{\bW}_s 
\label{eq:vn+1-1}
\end{align}
where $\bar{\bW}_s = \bW_{nh+s} - \bW_{nh}$. 
We rewrite these relations in the 
shorter form
\begin{align}
    \bvartheta_{n+U} &= \bvartheta_n + \gamma^{-1}\eta\big( U\bar\alpha_1
    \bv_n -  U^2\eta\bar\beta_1\,\bg_n
    +U\sqrt{2U\gamma\eta}\,\bar\sigma_1\bxi_1\big)
    \label{eq:thetaU}\\
    \bvartheta_{n+1} &= \bvartheta_n + \gamma^{-1}\eta\big(\bar\alpha_2 
    \bv_n - \eta\bar\beta_2\,\bg_{n+U} + \sqrt{2\gamma\eta}\,\bar\sigma_2 \bxi_2\big)
    \label{eq:thetan+1}\\
    \bv_{n+1} & = \bv_n -\eta \bar\alpha_2 \bv_n 
    - 2\eta\bar\beta_3 \bg_{n+U} + \sqrt{2\gamma\eta}
    \,\bar\sigma_3\bxi_3\label{eq:vn+1}
\end{align}
where $\bar\alpha_1$, $\bar\beta_1$, $\bar\beta_2$, $\bar\beta_3$ and $\bar\sigma_1$ are positive random variables
(with randomness inherited from $U$ only)
satisfying
\begin{align}
    \bar\alpha_1&\leqslant 1, \qquad
    \bar\beta_1 \leqslant 1/2,\qquad
    \bar\beta_2 \leqslant 1-U\leqslant 1,\qquad
    \bar\beta_3 \leqslant 1/2,\qquad
    \bar\sigma_1^2 \leqslant 1/3
\end{align}
and $\mathbb E[\bar\beta_2] \in[0.468, 0.5]$. 
Similarly, $\bar\alpha_2$, $\bar\sigma_2$ and $\bar 
\sigma_3$ are positive real numbers depending on 
$\gamma$ and $h$ such that 
\begin{align}
    \bar\alpha_2&\leqslant 1, \qquad
    \bar\sigma_2^2 \leqslant 1/3,\qquad
    \bar\sigma_3^2 \leqslant 1.
\end{align}
We define
\begin{align}
    \bar\bv_{n+1}:= \mathbb E_{U}[\bv_{n+1}],\qquad
    \bar\bvartheta_{n+1} := \mathbb E_{U}[\bvartheta_{n+1}]. 
\end{align}

The solution to SDE~\eqref{eq:uld} starting from 
$(\bv_n,\bvartheta_n)$ at the $n$-th iteration at 
time $h$ admits the following integral formulation
\begin{align}
    \bL'_t &=\bvartheta_n
    +\int_0^t\bV'_s\rmd s\\
    \bV'_t &=\bv_ne^{-\gamma t}
    -\gamma \int_0^t e^{-\gamma(t-s)}\nabla f(\bL'_s)\,\rmd s
    +\sqrt{2}\,\gamma\int_0^t e^{-\gamma(t-s)}\rmd\bW_{nh+s}\,.
    \label{eq:exactsol}
\end{align}
These expressions will be used in the proofs provided 
in the present section. Furthermore, without loss of 
generality, we assume that the $f(\btheta_*)= \min_{
\btheta\in\mathbb R^p} f(\btheta) = 0$.  

\subsection{Some preliminary results}

We start with some technical results required to prove~\Cref{thm:rklmc}. They mainly 
assess the discretisation error as well
as discounted sums of squared gradients and
velocities.

\begin{lemma}[Precision of the mid-point]
\label{lem:9}
For every  $h>0$, it holds that
\begin{align}
\|\bvartheta_{n+U}-\bL'_{U h}\|_{\Ltwo} 
\leqslant  \gamma^{-1}M_\gamma \eta^3e^{M_\gamma\eta^2/2}\Big(0.065 \eta \|\bg_n\|_{\Ltwo} + (1/6)\|\bv_n\|_{\Ltwo} 
+ \sqrt{\eta\gamma p /54}\Big)\,.
\end{align}
\end{lemma}

\begin{lemma}[Discretization error]
\label{lem:2}
Let $(\bL'_t,\bV'_t)$ be the exact solution of the 
kinetic Langevin diffusion starting from $(\bvartheta_n,\bv_n)$. If $\gamma \geqslant M$ 
and $h>0$, it holds that
\begin{align}
    \gamma\|\bar\bvartheta_{n+1} - \bL'_h\|_{\Ltwo} 
    &\leqslant \frac{M_\gamma^2 \eta^5e^{M_\gamma\eta^2/2}}{\sqrt{3}}
    \Big(0.065 \eta \|\bg_n\|_{\Ltwo} 
    + (1/6) \|\bv_n\|_{\Ltwo} + \sqrt{\eta\gamma p/54}
    \Big)\\
    \gamma\|\bvartheta_{n+1}-\bar\bvartheta_{n+1}\|_{
    \Ltwo}& \leqslant  M_\gamma \eta^3\big(
    0.26\|\bv_n\|_{\Ltwo}  + 0.106\sqrt{\eta\gamma p }
    \big) + \frac{\eta^2}{\sqrt3}\big(0.12{M_\gamma\eta^2}+1\big)\|\bg_n\|_{\Ltwo}\\
    \|\bar \bv_{n+1} -\bV'_h\|_{\Ltwo}
    & \leqslant M_\gamma^2\eta^4e^{M_\gamma\eta^2/2}\Big(0.065 \eta 
    \|\bg_n\|_{\Ltwo} + (1/6)\|\bv_n\|_{\Ltwo}  + \sqrt{\eta\gamma p /54}\Big) \\
    \|\bv_{n+1}-\bar\bv_{n+1}\|_{\Ltwo}
    & \leqslant M_\gamma \eta^2\big(0.82\|\bv_n\|_{\Ltwo} + 
    0.41\sqrt{\eta\gamma p }\big) + \frac{\eta^2}{\sqrt3 }\big(0.55 M_\gamma \eta + 1)\big\|\bg_n\big\|_{\Ltwo}.
\end{align}
\end{lemma}

\begin{corollary}
\label{cor:rklmc}
    If $\gamma \geqslant 2M$ and $\eta\leqslant 1/5$,
    it holds that
    \setlength\tabcolsep{1pt}
    \begin{center}
         \begin{tabular}{rrrrr}    
    $\gamma\|\bar\bvartheta_{n+1} - \bL'_h\|_{\Ltwo}\leqslant$ & $ M_\gamma^2\eta^5
    \big($ & $0.038 \eta\|\bg_n\|_{\Ltwo} +$ & $0.098\|\bv_n\|_{\Ltwo} +$ & $0.084\sqrt{\eta\gamma p}\big),$\\[2pt]
    $\gamma\|\bvartheta_{n+1}-\bar\bvartheta_{n+1}\|_{
    \Ltwo}\leqslant $ & $ \eta^2\big($ & $ 0.578\|\bg_n\|_{\Ltwo}+$ & $
    0.02 \|\bv_n\|_{\Ltwo}  + $ & $ 0.005
    \sqrt{\eta\gamma p }\big),$\\[2pt]
    $\|\bar \bv_{n+1} -\bV'_h\|_{\Ltwo}
     \leqslant $ & $M_\gamma^2\eta^4\big($ & $0.066 \eta \|\bg_n\|_{\Ltwo} + $ & $0.168\|\bv_n\|_{\Ltwo}
    + $ & $ 0.137\sqrt{\eta\gamma p }\big), $\\[2pt]
    $\|\bv_{n+1}-\bar\bv_{n+1}\|_{\Ltwo}
    \leqslant$ & $ \eta^2\big($ & $0.591\big\|\bg_n\big\|_{\Ltwo} +$ & 
    $0.164 \|\bv_n\|_{\Ltwo} +$ & $ 0.082
    \sqrt{\eta\gamma p }\big).$
    \end{tabular}
    \end{center}
\end{corollary}

\begin{proposition}\label{prop:4}
    If $\gamma^2\geqslant 5 M$ and $\eta \leqslant 
    1/5$, then, for any $n\in\mathbb N$, the iterates
    of the RKLMC satisfy
    \begin{align}
    \eta\sum_{k=0}^{n}\varrho^{n-k}
        \|\bv_{k}\|_{\Ltwo}^2 &\leqslant   18.8\varrho^{n}\gamma \mathbb E[f_0] +  
     3.92(x_n + 1.5\sqrt{\gamma p})^2 + \frac{10.6\gamma^2 p}{m},\\
     \eta\sum_{k=0}^{n}\varrho^{n-k}
        \|\bg_{k}\|_{\Ltwo}^2 &\leqslant
        21.7\varrho^{n}\gamma \mathbb E[f_0] +  
     4.88(x_n + 1.5\sqrt{\gamma p})^2 + \frac{11.2\gamma^2 p}{m},
\end{align}
where $\varrho=\exp(-mh)$ and $x_n = 
       \big(\norm{\bv_n- \bV_{nh}}^2_{\Ltwo}
       +
      \norm{\bv_n- \bV_{nh}+\gamma(  \bvartheta_n-\bL_{nh})}_{\Ltwo}^2\big)^{1/2}$.
\end{proposition}

\begin{proof}[Proof of \Cref{prop:4}]
We use the same shorthand notation as in the previous
proofs and assume without loss of generality that 
$\btheta_* = 0$. 
Let us define $z_k = \mathbb E[\bv_k^\top \bg_k]$, and
\begin{align}
    S_n(z) &:= \sum_{k=0}^{n}\varrho^{n-k}z_k,
    &S_n(g^2) &:= \sum_{k=0}^n\varrho^{n-k} \norm{\bg_k}^2_{\Ltwo}, \\    
    S_n(f) &:= \sum_{k=0}^{n}\varrho^{n-k}
        \mathbb E[f_{k}],
    &S_n(v^2) &:= \sum_{k=0}^{n}\varrho^{n-k}
        \|\bv_{k}\|_{\Ltwo}^2\,.
\end{align}

We will need the following lemma, the proof of
which is postponed. 
\begin{lemma}\label{lem:8}
    For any $\gamma>0$ and $h>0$ satisfying $\gamma \geqslant 5M$ and any 
    $\eta \leqslant 1/5$, the iterates of the
    randomized midpoint discretization of the kinetic 
    Langevin diffusion satisfy
    \begin{alignat}{5}
    \|\bv_{n+1}\|_{\Ltwo}^2 
        & \leqslant & (1 - 1.47\eta ) \| \bv_n \|_{\Ltwo}^2 
        - & &2\bar\alpha_2\eta  \mathbb E [\bv_n^\top \bg_{n}] +&& 2\eta^2 \|\bg_n\|_{\Ltwo}^2 +
        &\,2.12\gamma\eta  p
        \label{eq:help1}\\
    \mathbb E[\bv_{n+1}^\top\bg_{n+1}] 
        &\leqslant &0.51 \eta \|\bv_n\|_{\Ltwo}^2  +&&\, (1- \bar\alpha_2\eta )\mathbb E[\bv_{n}^\top\bg_{n}] -&& 0.97\eta
        \|\bg_n\|_{\Ltwo}^2  +  &\,0.9  \eta^2\gamma p
        \label{eq:help2}\\
    \gamma \mathbb E[f_{n+1} - f_n] 
        &\leqslant &0.28\eta^2 \|\bv_n\|^2_{\Ltwo} +& &\bar\alpha_2\eta \mathbb
        E[\bv_n^\top \bg_n] -&& \,0.46\eta^2\|\bg_n\|^2_{\Ltwo}  +&\, 
        0.09 \eta^3\gamma  p.\label{eq:help3}
    \end{alignat}
\end{lemma}
From the first inequality~\eqref{eq:help1} in 
\Cref{lem:8}, we infer that
\begin{align}
    S_{n}^{+1}(v^2) &\leqslant (1 - 1.47\eta ) S_n(v^2) 
    -2\bar\alpha_2\eta  S_n(z) + 2\eta^2 S_n(g^2) 
    + 2.12\gamma^2p/m.
\end{align}
In view of \Cref{lem:rho-sum} and the fact that 
$\|\bv_0\|_{\Ltwo}^2 = \gamma p$, this implies that
\begin{align}
    1.47\eta S_n(v^2) + 2\bar\alpha_2\eta  S_n(z) 
     &\leqslant  S_n(v^2) - S_{n}^{+1}(v^2) + 2\eta^2 
     S_n(g^2) + 2.12\gamma^2p/m \\
     &\leqslant (1-\varrho) S_n(v^2) + 2\eta^2 
     S_n(g^2) + {2.12\gamma^2p}/{m} + \gamma p .
\end{align}
Note that $1-\varrho\leqslant \frac{m\eta}{\gamma} 
\leqslant 0.02\eta$. Therefore, we obtain
\begin{align}
    (1.47 - 0.02) S_n(v^2) + 2\bar\alpha_2 
    S_n(z) \leqslant 2\eta S_n(g^2) +  
    \frac{2.14\gamma^2 p}{m\eta},
\end{align}
that is equivalent to
\begin{align}
    \boxed{S_n(v^2)  \leqslant 1.38\bar\alpha_2 S_n(z)_- 
    + 1.38\eta S_n(g^2) +  \frac{1.48\gamma^2 p}{m\eta}.}
    \label{eq:54}
\end{align}
The second step is to use the second inequality~\eqref{eq:help2} of \Cref{lem:8}.
Note that $m\eta/\gamma \leqslant 1/500$ implies $1-\varrho\geqslant0.998 m\eta/\gamma$. 
It then follows that
\begin{align}
    S_{n}^{+1} (z)
        & = (1- \bar\alpha_2\eta )
        S_n(z) - 0.97\eta  S_n(g^2) + 
        0.51 \eta S_n(v^2) + {0.9 \eta\gamma^2 p/m}.
\end{align}
This  inequality, combined with \Cref{lem:rho-sum},  yields
\begin{align}
0.97 \eta S_n(g^2)
    &\leqslant -(\bar\alpha_2\eta +\varrho -1)S_n(z)  + 0.51\eta S_n(v^2)  + |z_{n+1}|
    + 0.9\eta\gamma^2 p/m\\
    &\leqslant \bar\alpha_2\eta S_n(z)_-  + 0.51\eta S_n(v^2)  + |z_{n+1}| +
    0.9\eta\gamma^2 p/m.
\end{align}
This can be rewritten as 
\begin{align}
    \boxed{S_n(g^2)
    \leqslant  1.03 \bar\alpha_2 S_n(z)_- + 
    0.53 S_n(v^2) + \frac{1.03|z_{n+1}|}\eta
    + \frac{0.93 \gamma^2 p}{m}.}
    \label{eq:xx}
\end{align}
Let us now proceed with a similar treatment for the
last inequality of \Cref{lem:8}. Applying 
\Cref{lem:rho-sum}, we get $S_{n}^{+1}(f)\geqslant 
\varrho S_n(f) - \varrho^{n+1} \mathbb E[f_0]\geqslant 
(1 - m \eta/\gamma ) S_n(f) - \varrho^{n+1} 
\mathbb E[f_0]$, which leads to
\begin{align}
    -m\eta S_n(f)
        &\leqslant\varrho^{n+1} \gamma\mathbb E[f_0] 
        + 0.28 \eta^2 S_n (v^2) + \bar\alpha_2 \eta  
        S_n(z) - 0.46\eta^2 S_n(g^2) + 0.09 
        \frac{\eta^2\gamma^2 p}{m}.
\end{align}
From this inequality, and the Polyak-Lojasievicz 
condition, one can infer that
\begin{align}\label{eq:Snz2}
    \boxed{
    \bar\alpha_2 S_n(z)_-
        \leqslant \varrho^{n+1} \gamma \mathbb E
        [f_0]/\eta + 0.28 \eta S_n(v^2) + (0.5 - 
        0.46\eta) S_n(g^2) + 0.09 \frac{\eta\gamma^2 
        p}{m}.}
\end{align}
Combining \eqref{eq:Snz2} with \eqref{eq:54},
we get
\begin{align}
    S_n(v^2) \leqslant 1.38 \Big(\varrho^{n+1} 
    \gamma\mathbb E[f_0]/\eta &+ 0.28 \eta S_n(v^2) + 
    (0.5 - 0.46\eta) S_n(g^2) + 0.09 \frac{\eta\gamma^2 p}{m}\Big)\\
    & + 1.38\eta S_n(g^2)  +\frac{1.48\gamma^2 p
    }{m\eta}.
\end{align}
Since $\eta\leqslant 0.1$, it follows then
\begin{align}\label{eq:h1}   
     \boxed{S_n(v^2) \leqslant 0.8\Big( S_n(g^2) + \frac{1.8 \varrho^{n}\gamma}{
     \eta} \,\mathbb E[f_0] +  
     \frac{2\gamma^2p}{m\eta}\Big)\label{eq:55}\,.}
\end{align}
Similarly, combining \eqref{eq:Snz2} and
\eqref{eq:xx}, we get
\begin{align}
    S_n(g^2) \leqslant 1.03 \Big(\varrho^{n}\gamma 
    \mathbb E[f_0]/\eta &+ 0.28 \eta S_n(v^2) + 
    (0.5 - 0.46\eta) S_n(g^2) + 0.09 \frac{\eta\gamma^2 p}{m}\Big)\\
    & + 0.53 S_n(v^2) + \frac{1.03|z_{n+1}|}{\eta} + \frac{0.93\gamma^2 p
    }{m}.
\end{align}
Since $\eta\leqslant 0.1$, it follows then
\begin{align}\label{eq:h2}
    \boxed{
    S_n(g^2) \leqslant 1.05  S_n(v^2) + \frac{ 1.94\varrho^{n}\gamma}{\eta}\, 
    \mathbb E[f_0] +  
     \frac{1.94|z_{n+1}|}{\eta} + \frac{0.94\gamma^2 p}{m}.
    }
\end{align}
Equations \eqref{eq:h1} and \eqref{eq:h2} 
together yield 
\begin{align}
    S_n(g^2) &\leqslant 0.84 \Big( 
    S_n(g^2) + \frac{1.8 \varrho^{n}\gamma}{
     \eta} \,\mathbb E[f_0] +  
     \frac{2\gamma^2p}{m\eta}\Big) 
     + \frac{ 1.94\varrho^{n}\gamma}{\eta}\, 
    \mathbb E[f_0] +  
     \frac{1.94|z_{n+1}|}{\eta} + \frac{0.94\gamma^2 p}{m}\\
     &\leqslant 0.84  
    S_n(g^2) 
     + \frac{ 3.46\varrho^{n}\gamma}{\eta}\, 
    \mathbb E[f_0] +  
     \frac{1.94|z_{n+1}|}{\eta} + \frac{1.78\gamma^2 p}{m\eta}.
\end{align}
Hence, we get
\begin{align}
    \boxed{
    S_n(g^2) \leqslant \frac{ 21.7\varrho^{n}\gamma}{\eta}\, \mathbb E[f_0] +  
     \frac{12.2|z_{n+1}|}{\eta} + \frac{11.2\gamma^2 p}{m\eta}.}
\end{align}
Using once again equation \eqref{eq:h1}, 
we arrive at 
\begin{align}
    S_n(v^2) &\leqslant 0.8\Big( \frac{ 21.7\varrho^{n}\gamma}{\eta}\, \mathbb E[f_0] +  
     \frac{12.2|z_{n+1}|}{\eta} + \frac{11.2\gamma^2 p}{m\eta} + \frac{1.8 \varrho^{n+1}}{
     \eta} \,\mathbb E[f_0] +  
     \frac{2\gamma^2p}{m\eta}\Big),
\end{align}
which is equivalent to
\begin{align}
    \boxed{
    S_n(v^2) \leqslant  \frac{ 18.8\varrho^{n}\gamma}{\eta}\, \mathbb E[f_0] +  
     \frac{9.8|z_{n+1}|}{\eta} + \frac{10.6\gamma^2 p}{m\eta}.}
\end{align}
To complete the proof of the proposition, 
it remains to establish the suitable upper
bound on $|z_{n+1}|$. To this end, we note
that 
\begin{align}
\norm{\bg_n}_{\Ltwo} &\leqslant 
    M \norm{\bvartheta_n-\bL_{nh}}_{\Ltwo}+\sqrt{M p}\\
    &\leqslant 0.2\norm{\gamma(\bvartheta_n-\bL_{nh})}_{\Ltwo} + \sqrt{0.2\gamma p}\\
    &\leqslant 0.3 (x_n + 1.5\sqrt{\gamma p})\\
\norm{\bv_n}_{\Ltwo}
    &\leqslant \norm{\bv_n-\bV_{nh}}_{\Ltwo}+\sqrt{\gamma p}\\
    &\leqslant \norm{\bv_n-\bV_{nh}}_{\Ltwo}+\sqrt{\gamma p}\\
    &\leqslant x_n +\sqrt{\gamma p}.
\end{align}
Then,  following the same steps as those
used in the proof of the second inequality of \Cref{lem:8}, one can infer that 
\begin{align}
    |z_{n+1}| 
    &\leqslant  
        |z_n| + 0.97\eta\|
        \bg_n\|_{\Ltwo}^2 +  0.51 \eta\|\bv_n\|_{
        \Ltwo}^2 +  0.09  \eta^2\gamma p\\
    &\leqslant  \|
        \bg_n\|_{\Ltwo}\|
        \bv_n\|_{\Ltwo}
         + 0.1\|\bg_n\|_{\Ltwo}^2 +  0.051 \|\bv_n\|_{
        \Ltwo}^2 +  0.001 \gamma p\\
    &\leqslant 1.1\|\bg_n\|^2_{\Ltwo} + 
    0.301\|\bv_n\|^2_{\Ltwo} + 0.001\gamma p\\
    &\leqslant 0.099(x_n + 1.5\sqrt{\gamma p})^2 
        + 0.301(x_n + 1.1\sqrt{p})^2\\
    &\leqslant 0.4\big(x_n +1.5 \sqrt{p}\big)^2.
\label{eq:help4}
\end{align}
This completes the proof of the proposition. 

\end{proof}

\subsection{Proof of \Cref{thm:rklmc}}

Let $\bvartheta_{n+U},\bvartheta_{n+1},\bv_{n+1}$ 
be the iterates of Algorithm.
Let $(\bL_t,\bV_t)$ be the kinetic Langevin
diffusion, coupled with $(\bvartheta_n,\bv_n)$ through the same Brownian motion and starting from a random point $(\bL_0,\bV_0)\propto \exp(-f (\btheta) - \frac{1}{2}\|\bv\|^2)$ such that $\bV_0 = \bv_0$. Let $(\bL'_t,\bV'_t)$ be the kinetic 
Langevin diffusion defined on $[0,h]$ using the same Brownian motion and starting from $(\bvartheta_n, 
\bv_n)$. 

Our goal will be to bound the term $x_n$ defined by
\begin{align}
    x_n = \bigg\| \bfC
    \begin{bmatrix}
        \bv_n- \bV_{nh}\\
        \bvartheta_n-\bL_{nh}
    \end{bmatrix}
    \bigg\|_{\Ltwo}
    \quad \text{with}\quad 
    \bfC = \begin{bmatrix}
    \mathbf I_p & \mathbf 0_p\\
    \mathbf I_p & \gamma\mathbf I_p
    \end{bmatrix}.
\end{align}
To this end, define
\begin{align}
    \bar\bv_{n+1} = \mathbb E_{U}[\bv_{n+1}],\qquad
    \bar\bvartheta_{n+1} = \mathbb E_{U}[\bvartheta_{n+1}]. 
\end{align}
Since $(\bV_{(n+1)h},\bL_{(n+1)h})$ are independent
of $U$, we have 
\begin{align}
    x_{n+1}^2 & = \bigg\| \bfC
    \begin{bmatrix}
        \bv_{n+1}- \bar\bv_{n+1}\\
        \bvartheta_{n+1}-\bar\bvartheta_{n+1}
    \end{bmatrix}
    \bigg\|_{\Ltwo}^2 + 
    \bigg\| \bfC
    \begin{bmatrix}
        \bar\bv_{n+1}- \bV_{(n+1)h}\\
        \bar\bvartheta_{n+1}-\bL_{(n+1)h}
    \end{bmatrix}
    \bigg\|_{\Ltwo}^2.
\end{align}
Using the triangle inequality and Proposition~\ref{prop:klmc_contr} (See also Proposition 1 from~\citep{dalalyan_riou_2018}), we get
\begin{align}
    \bigg\| \bfC
    \begin{bmatrix}
        \bar\bv_{n+1}- \bV_{(n+1)h}\\
        \bar\bvartheta_{n+1}-\bL_{(n+1)h}
    \end{bmatrix}
    \bigg\|_{\Ltwo}&\leqslant 
    \bigg\| \bfC
    \begin{bmatrix}
        \bar\bv_{n+1}- \bV'_{h}\\
        \bar\bvartheta_{n+1}-\bL'_{h}
    \end{bmatrix}
    \bigg\|_{\Ltwo}   
    +
    \bigg\| \bfC
    \begin{bmatrix}
        \bar\bV'_{h}- \bV_{(n+1)h}\\
        \bar\bL'_{h}-\bL_{(n+1)h}
    \end{bmatrix}
    \bigg\|_{\Ltwo}\\
    &\leqslant \bigg\| \bfC
    \begin{bmatrix}
        \bar\bv_{n+1}- \bV'_{h}\\
        \bar\bvartheta_{n+1}-\bL'_{h}
    \end{bmatrix}
    \bigg\|_{\Ltwo}   
    + \varrho x_{n}
\end{align}
where $\varrho = e^{-mh}$. Combining these
inequalities, we get
\begin{align}
    x_{n+1}^2 &\leqslant \big(\varrho x_{n} + y_{n+1}\big)^2 + z_{n+1}^2 
\end{align}
where
\begin{align}
    y_{n+1} = \bigg\| \bfC
    \begin{bmatrix}
        \bar\bv_{n+1} - \bV'_{h}\\
        \bar\bvartheta_{n+1} - \bL'_{h}
    \end{bmatrix}
    \bigg\|_{\Ltwo},\qquad 
    z_{n+1} = \bigg\| \bfC
    \begin{bmatrix}
        \bv_{n+1}- \bar\bv_{n+1}\\
        \bvartheta_{n+1}-\bar\bvartheta_{n+1}
    \end{bmatrix}
    \bigg\|_{\Ltwo}^2.
\end{align}
This yields\footnote{One can check by induction, 
that if for some sequences $x_n,y_n,z_n$ and 
some $\varrho \in(0,1)$ it holds that $x_{n+1}^2
\leqslant  (\varrho x_{n} + y_{n+1})^2 + z_{n+1
}^2$, then necessarily $x_n\leqslant \varrho^n 
x_0 + \sum_{k=1}^n \varrho^{n-k}y_k + (\sum_{
k=1}^n \varrho^{2(n-k)}z_k^2)^{1/2}$ for every 
$n\in\mathbb N$.}
\begin{align}
    x_n &\leqslant \varrho^n x_0 + \sum_{k=1}^n 
    \varrho^{n-k} y_k + \bigg(\sum_{k=1}^n
    \varrho^{2(n-k)} z_k^2\bigg)^{1/2}\\
    &\leqslant \varrho^n x_0 + \bigg(\frac{1}{
    {1-\varrho}}\sum_{k=1}^n \varrho^{n-k} y_k^2 
    \bigg)^{1/2} + \bigg(\sum_{k=1}^n \varrho^{
    2(n-k)} z_k^2\bigg)^{1/2},
    \label{eq:xnC}
\end{align}
where the second inequality follows from
the Cauchy-Schwarz inequality and the formula
of the sum of a geometric progression. 
Using the fact that $\|\bfC[a, b]^\top\|^2 = 
\|a\|^2 + \|a + \gamma b\|^2 \leqslant 3\|a\|^2 
+ 2\gamma^2\|b\|^2$, we arrive at
\begin{align}
    y_{n+1}^2&\leqslant 3\|\bar\bv_{n+1} - \bV'_h 
    \|_{\Ltwo}^2 + 2\gamma^2 \|\bar\bvartheta_{
    n+1} - \bL'_h\|_{\Ltwo}^2, \label{eq:yn}\\
    z_{n+1}^2&\leqslant 3\|\bv_{n+1} - \bar\bv_{
    n+1} \|_{\Ltwo}^2 + 2\gamma^2 \|\bvartheta_{
    n+1} - \bar\bvartheta_{n+1}\|_{\Ltwo}^2.
    \label{eq:zn}
\end{align}
We then have
\begin{align}
   x_n     
   &\leqslant \varrho^n x_0
   +\bigg(\frac{1.001\gamma}{m\eta}\sum_{k=1}^n 
   \varrho^{n-k} (3\|\bar\bv_{k} - \bV'_h 
    \|_{\Ltwo}^2 + 2\gamma^2 \|\bar\bvartheta_{k} 
    - \bL'_h\|_{\Ltwo}^2) \bigg)^{1/2}\\
   &\qquad +\bigg(\sum_{k=1}^n
    \varrho^{2(n-k)} (3\|\bv_{k} - \bar\bv_{k} 
    \|_{\Ltwo}^2 + 2\gamma^2 \|\bvartheta_{k} -
    \bar\bvartheta_{k}\|_{\Ltwo}^2)\bigg)^{1/2}\,.
    \label{eq:xn_rkl}
\end{align}
By Corollary~\ref{cor:rklmc}, we find
\begin{align}
    \|\bar\bv_{k} - \bV'_h \|_{\Ltwo}^2     
        &\leqslant M_\gamma^4\eta^8\Big(0.066 
        \eta \|\bg_{k-1}\|_{\Ltwo} + 0.168 
        \|\bv_{k-1}\|_{\Ltwo} +  0.137\sqrt{
        \eta\gamma p }\Big)^2\\
        &\leqslant  0.2^3 M_\gamma\eta^8\times 
        0.0514\big(\eta^2 \|\bg_{k-1}\|_{\Ltwo}^2 
        + \|\bv_{k-1}\|_{\Ltwo}^2 + \eta\gamma p 
        \big),\\
    \gamma^2 \|\bar\bvartheta_{k} - 
        \bL_h'\|_{\Ltwo}^2 &\leqslant  M_\gamma^4
        \eta^{10} \Big(0.038 \eta\|\bg_{k-1}\|_{
        \Ltwo} + 0.098\|\bv_{k-1}\|_{\Ltwo} + 
        0.084\sqrt{\eta\gamma p}\Big)^2\\
        &\leqslant 0.2^3M_\gamma\eta^{8}\times 
        0.0002\big(\eta^2 \|\bg_{k-1} 
        \|_{\Ltwo}^2 + \|\bv_{k-1}\|_{\Ltwo}^2
        +  \eta\gamma p \big)\\
    \|\bv_k - \bar\bv_{k}\|_{\Ltwo}^2 
        &\leqslant \eta^4\Big(0.591  \|\bg_{k-1}
        \|_{\Ltwo} + 0.164\|\bv_{k-1}\|_{\Ltwo} 
        + 0.082\sqrt{\eta\gamma p }\Big)^2\\
        &\leqslant   \eta^4\times 0.39\big( 
        \|\bg_{k-1}\|_{\Ltwo}^2 + \|\bv_{k-1}
        \|_{\Ltwo}^2 + \eta\gamma p \big),\\
    \gamma^2 \|\bvartheta_k-\bar\bvartheta_{k} 
        \|_{\Ltwo}^2 &\leqslant  \eta^4 \Big( 
        0.578 \|\bg_{k-1}\|_{\Ltwo} + 0.02\|
        \bv_{k-1}\|_{\Ltwo} + 0.005\sqrt{\eta
        \gamma p}\Big)^2\\
    &\leqslant \eta^{4}\times 0.32\big(\|
        \bg_{k-1}\|_{\Ltwo}^2 + \|\bv_{k-1}
        \|_{\Ltwo}^2 + \eta\gamma p \big)
\end{align}
Therefore, we infer from \eqref{eq:xn_rkl} 
that
\begin{alignat}{2}
    x_n &\leqslant    \varrho^n x_0 &+& 
        \bigg(\frac{\gamma}{m\eta}\sum_{k=0}^n 
        0.2^2 M_\gamma\eta^8 \times 0.031 
        \varrho^{n-k}\big(\eta^2 \|\bg_k 
        \|_{\Ltwo}^2 + \|\bv_k\|_{\Ltwo}^2
        +  \eta\gamma p \big)  \bigg)^{1/2}\\
        & &+& \bigg(\sum_{k=0}^n 1.82\eta^4 
        \varrho^{2(n-k)}\big(\norm{\bg_k}^2_{
        \Ltwo} + \norm{\bv_k}^2_{\Ltwo} + 
        \eta\gamma p\big)  \bigg)^{1/2}\,.
\end{alignat}
From \Cref{prop:4} it then follows that
\begin{align}
\eta\sum_{k=0}^n \varrho^{n-k}\big(\eta^2 
    \|\bg_k\|_{\Ltwo}^2 + \|\bv_k\|_{\Ltwo}^2
    +  \eta\gamma p \big)
    &\leqslant 18.9\varrho^{n}\gamma \mathbb 
    E[f_0] +  3.97(x_n + 1.5\sqrt{\gamma p})^2 
    + \frac{10.8\gamma^2 p}{m},\\
\eta\sum_{k=0}^n \varrho^{2(n-k)}\big(
     \norm{\bg_k}^2_{\Ltwo} + \norm{\bv_k}^2_{
     \Ltwo} + \eta\gamma p\big) &\leqslant 
     40.5\varrho^{n}\gamma \mathbb E[f_0] +  
     8.8(x_n + 1.5\sqrt{\gamma p})^2 + 
     \frac{21.9\gamma^2 p}{m}.
\end{align}
This yields
\begin{alignat}{2}
    x_n &\leqslant    \varrho^n x_0 &+& 0.036
        \eta^3\sqrt{\kappa}\Big( 18.9\varrho^{n} 
        \gamma \mathbb E[f_0] +  3.97(x_n + 1.5
        \sqrt{\gamma p})^2 + \frac{10.8\gamma^2 
        p}{m}\Big)^{1/2}\\
        & &+& \eta^{3/2}\Big(74\varrho^{n}\gamma
        \mathbb E[f_0] + 16(x_n + 1.5\sqrt{\gamma 
        p})^2 + \frac{40\gamma^2 p}{m}\Big)^{1/2}\,\\
        &\leqslant \varrho^n x_0 & +&  (0.072 
        \eta^3\sqrt{\kappa} + 4\eta^{3/2}) x_n + 
        (0.16\eta^3\sqrt{\kappa} + 8.7\eta^{3/2}) 
        \sqrt{\varrho^n\gamma\mathbb E[f_0]}\\
        & & + & 0.12\eta^3\gamma\sqrt{\kappa p/m} 
        + 6.4\eta^{3/2} \gamma \sqrt{p/m}.
\end{alignat}
We assume that $\eta\kappa^{1/6}\leqslant 0.1$, 
which implies that
\begin{align}
    x_n &\leqslant    \varrho^n x_0  +0.072 x_n 
    + 0.16 \sqrt{\varrho^n\gamma\mathbb E[f_0]} 
    +  0.12\eta^3\gamma\sqrt{\kappa p/m} 
    + 6.4\eta^{3/2} \gamma \sqrt{p/m}.
\end{align}
Rearranging the display leads to
\begin{align}
x_n &\leqslant    1.08\varrho^n x_0 + 0.18 
    \sqrt{\varrho^n\gamma\mathbb E[f_0]} + 
    0.12\eta^3\gamma\sqrt{\kappa p/m} 
    + 6.9\eta^{3/2} \gamma \sqrt{p/m}.
\end{align}
Finally, we use the fact that $x_0 = \gamma 
\wass_2(\nu_0,\pi)$ and $x_n\geqslant \gamma 
\wass_2(\nu_n,\pi)/\sqrt{2}$ to get the 
claim of the theorem.

\subsection{Proofs of the technical lemmas}

We now provide the proofs of the technical 
lemmas that we used in this section.

\subsubsection{Proof of Lemma~\ref{lem:9}}

By the definition of $\bvartheta_{n+U}$, we have
\begin{align}
    \|\bvartheta_{n+U} -  \bL'_{U h}\| 
    &\leqslant    \bigg\|\int_0^{
    U h}\big(1-e^{-\gamma(U h -s)}\big)\big(
    \nabla f(\bvartheta_n)-\nabla f(\bL'_s)\big) 
    \rmd s\bigg\|\\
    &\leqslant \int_0^{U h}
    \Big\|\big(1-e^{-\gamma(U h -s)}\big)\big(
    \nabla f(\bL'_0)-\nabla f(\bL'_s)\big)\Big \|
    \,\rmd s\\ 
    & = Uh  \int_0^1  \big(1- e^{-
    U\eta(1-t)}\big)\big\|\nabla f(\bL'_0) 
    - \nabla f(\bL'_{Uht})\big\|\,\rmd t\\
    & \leqslant Mh\eta U^2  \int_0^1 (1-t)
    \big\|\bL'_0 - \bL'_{Uht}\big\|\,\rmd t,
\end{align}
where in the last inequality we have used the
Lipschitz property of $\nabla f$ and the 
inequality $1-e^{-U\eta(1-t)}\leqslant 
U\eta(1 -t)$. By taking the expectation
wrt to $U$, we get
\begin{align}
    \mathbb E_U\|\bvartheta_{n+U} -  \bL'_{U h}
    \|^2 &\leqslant M^2 h^2\eta^2 \mathbb E_U\bigg[ 
    U^4  \bigg\{\int_0^1 (1-t) \big\| \bL'_0 
    - \bL'_{Uht}\big\|\,\rmd t\bigg\}^2\bigg]\\
    &\leqslant \frac{M^2h^2\eta^2}{3} \mathbb E_U\bigg[ 
    U^4  \int_0^1 \big\| \bL'_0 
    - \bL'_{Uht}\big\|^2\,\rmd t\bigg]\\
    &\leqslant \frac{M^2h^2\eta^2}{3} \mathbb E_U [U^3] 
    \int_0^1 \big\| \bL'_0 - \bL'_{ht}\big\|^2\,
    \rmd t.
\end{align}
Hence, we obtain in view of \cref{lem21:5}
\begin{align}
    \|\bvartheta_{n+U} - \bL'_{U h}\|_{\Ltwo}^2 
    &\leqslant  \frac{M^2h^2\eta^2}{12} \int_0^1 
    \|\bL'_0-\bL'_{ht} \|_{\Ltwo}^2\,\rmd t\\
    &\leqslant  \frac{M^2h^2\eta^2 e^{M_\gamma\eta^2}}{12} \int_0^1 
    \bigg(\sqrt{\frac{2\gamma^2 p(ht)^3}{3}} + ht\|\bv_n
    \|_{\Ltwo} + \frac{\gamma(ht)^2}{2}\|\nabla f(\bvartheta_n)
    \|_{\Ltwo}\bigg)^2\rmd t \\
    & \leqslant \frac{\gamma^{-2}M_\gamma^2\eta^6 e^{M_\gamma\eta^2}}{12} 
    \bigg\{\sqrt{(2/3)\eta\gamma p} + \sqrt{1
    /3}\,\|\bv_n\|_{\Ltwo} + \sqrt{0.05} \gamma h
    \|\nabla f(\bvartheta_n)\|_{\Ltwo}\bigg\}^2.
\end{align}
Taking the square root of the two sides of the
inequality, we get the claim of the lemma.

\subsubsection{Proof of \Cref{lem:2}}

By the definition of $\bvartheta_{n+1},$ we 
have
\begin{align}
    \|\bar\bvartheta_{n+1}-\bL'_h\| 
    & = \Big\|\mathbb E_U \Big[h
    \big(1-e^{-\gamma(h-U h)}\big)\nabla f 
    (\bvartheta_{n+U})\Big]- 
    \int_0^h\big(1-e^{-\gamma(h-s)}\big)\nabla
    f(\bL'_s)\,\rmd s\Big\| \\
    & = \Big\|\mathbb E_U \Big[h
    \big(1-e^{-\gamma(h-U h)}\big)\nabla f_{n+U}\Big]-h
    \mathbb E_U \Big[\big(1-e^{-\gamma(h-hU)}
    \big)\nabla f(\bL'_{U h})\Big]\Big\| \\
    &\leqslant  h
    \mathbb E_U\Big[\big(1-e^{-\gamma(1-U)
    h} \big)\|\nabla f_{n+U}-\nabla f(\bL'_{U h})\| \Big]\\
    &\leqslant  M_\gamma\eta^2 \mathbb E_U\Big[(1- U)\|
    \bvartheta_{n+U} - \bL'_{U h}\| \Big],
\end{align}
where in the last inequality follows from the 
smoothness of function $f$ and the fact that 
$1-e^{-\gamma(h-U h)}\leqslant \gamma (1 - U)h$. 
Using the Cauchy-Schwarz inequality, we get
\begin{align}
    \|\bar\bvartheta_{n+1}-\bL'_h\|^2 
    & \leqslant  M_{\gamma}^2\eta^4 \mathbb E_U[(1 - U )^2]
    \,\mathbb E_U\big[\|\bvartheta_{n+U} - 
    \bL'_{U h}\|^2 \big]\\
    & = \frac{M_{\gamma}^2\eta^4}{3} \mathbb E_U\big[\|
    \bvartheta_{n+U} - \bL'_{U h} \|^2 \big].
\end{align}
By \Cref{lem:9}, we then obtain
\begin{align}
\|\bar\bvartheta_{n+1}-\bL'_h\|_{\Ltwo} 
&\leqslant  \frac{M_\gamma\eta^2}{\sqrt{3}}  \|
\bvartheta_{n+U} - \bL'_{U h} \|_{\Ltwo}\\
&\leqslant \frac{M_{\gamma}^2 \eta^5e^{M_\gamma\eta^2/2}}{\sqrt{3}\gamma}
\Big(0.065 \eta \|\bg_n\|_{\Ltwo} 
+ (1/6) \|\bv_n\|_{\Ltwo} + (1/7) \sqrt{
\eta\gamma p }\Big). 
\end{align}
This completes the proof of the first claim.

Using the definition of $\bvartheta_{n+1}$, and 
the fact that the mean minimizes the squared 
integrated error, we get
\begin{align}
\|\bvartheta_{n+1}-\bar\bvartheta_{n+1}\|_{\Ltwo}
& = h\Big\|\big(1 - 
e^{-\gamma h(1-U) }\big)\nabla f_{n+U} 
- \mathbb E_U\big[\big(1 - 
e^{-\gamma h(1-U)}\big)\nabla f_{n+U}\big]\Big\|_{\Ltwo}\\
&\leqslant   h \Big\|\big(1 - 
e^{-\gamma h(1 - U)}\big)\nabla f_{n+U} 
- \mathbb E_U\big[1 -  e^{-\gamma h(1-U)}\big]\nabla f_{n}\Big\|_{\Ltwo}.
\end{align}
Recall that $\bar U = 1-U$, combining this with the last display and the triangle inequality yields
\begin{align}
\|\bvartheta_{n+1}-\bar\bvartheta_{n+1}\|_{\Ltwo}
&\leqslant   h\Big\|\big(1 - 
e^{-\eta\bar U}\big)
\big(\nabla f_{n+U} - \nabla f_n\big)\Big\|_{\Ltwo}
    + h\Big\|\big( 
e^{-\eta \bar U} - \mathbb E[e^{-\eta \bar U}]\big)\nabla f_{n}\Big\|_{\Ltwo}\\
    &\leqslant M_\gamma\eta^2\big\|\bar  U(\bvartheta_{n+U} - \bvartheta_n)\big\|_{\Ltwo} +  h \eta \|\bar U\|_{\Ltwo}
    \|\bg_n\|_{\Ltwo}.\label{eq:41}
\end{align}
In view of \eqref{eq:thetaU}, we get
\begin{align}
    \big\|(1 - U)(\bvartheta_{n+U} - \bvartheta_n)
    \big\|_{\Ltwo}^2 &= \mathbb E\Big[(1-U)^2
    \Big( \| (\eta/\gamma)(U\bar\alpha_1 \bv_n - U^2\eta 
    \bar\beta_1 \bg_n)\|^2
    + (2/3) U^3 \eta^3 p/\gamma \Big)\Big]\\
    &\leqslant \frac{\eta^2\|\bv_n\|_{\Ltwo}^2}{15\gamma^2} + 
    \frac{\eta^4\|\bg_n
    \|^2_{\Ltwo}}{210\gamma^2} + \frac{\eta^3 p}{90\gamma}.
\end{align}
In addition, $\|1-U\|_{\Ltwo} = \sqrt{1/3}$. Therefore,
we infer from \eqref{eq:41} that 
\begin{align}
    \|\bvartheta_{n+1}-\bar\bvartheta_{n+1}\|_{\Ltwo} 
    &\leqslant\frac{M_\gamma \eta^3}{\gamma}\Big(
    \frac{\|\bv_n\|_{\Ltwo}}{\sqrt{15}}  + \sqrt{\frac{\eta\gamma p}{90}}
    \Big) + \frac{\eta^2}{\gamma}\Big(\frac{M_{\gamma}\eta^2}{\sqrt{210}} +\frac{1}{\sqrt{3}} \Big)\|\bg_n\|_{\Ltwo}. 
\end{align}
Numerical computations complete the proof of
the second claim.

By the definition \eqref{eq:vn+1} 
of $\bv_{n+1},$ we have
\begin{align}
\|\bar\bv_{n+1}-\bV'_{h}\|_{\Ltwo}
&=\gamma \bigg\|\mathbb E_U \big[he^{-\gamma(h-U h)} \nabla f_{n+U}\big] - \int_0^h e^{-\gamma(t-s)}\nabla f(\bL'_s)\,\rmd s\bigg\|_{\Ltwo}\\
&\leqslant   \gamma\Big\|he^{-\gamma(h-U h)}\nabla f_{n+U}-he^{-\gamma(h-U h)}\nabla f(\bL'_{U h})\Big\|_{\Ltwo}\\
&\leqslant M \gamma h \|\bvartheta_{n+U} - \bL'_{U h} \|_{\Ltwo}.
\end{align}
By \Cref{lem:9}, we obtain the third claim of the lemma.

In view of \eqref{eq:vn+1-1}, and the fact that the expectation minimizes the mean squared error, we have
\begin{align}
    \|\bv_{n+1}-\bar\bv_{n+1}\|_{\Ltwo}
    & = \gamma {h}\Big\|
e^{-\gamma h(1-U)}\nabla f_{n+U} 
- \mathbb E_U\big[ 
e^{-\gamma h(1-U)}\nabla f_{n+U}\big]\Big\|_{\Ltwo}\\
&\leqslant   \gamma h\Big\| 
e^{-\gamma h(1 - U)}\nabla f_{n+U} 
- \mathbb E_U\big[ e^{-\gamma h(1-U)}\big]\nabla f_{n}\Big\|_{\Ltwo}.
\end{align}
The last display, the notation $\bar U = 1 - U$ 
and the triangle inequality imply that 
\begin{align}
    \|\bv_{n+1}-\bar\bv_{n+1}\|_{\Ltwo}
    &\leqslant \gamma h\big\| e^{-\eta\bar U} 
    \big(\nabla f_{n+U} - \nabla f_{n}\big)
    \big\|_{\Ltwo} 
    +\gamma h\big\| \big( e^{-\eta\bar U} - \mathbb E_U[e^{-\eta \bar 
    U}]\big) \nabla f_{n}\big\|_{\Ltwo}\\
    &\leqslant  M \gamma h\big\|\bvartheta_{n+U} 
    - \bvartheta_{n}\big\|_{\Ltwo} + \gamma h\big\| 
    \big( e^{-\eta(1 - U)} - 1\big) 
    \nabla f_{n}\big\|_{\Ltwo}\\
    &\leqslant M \gamma h\big\|\bvartheta_{n+U}
    - \bvartheta_{n}\big\|_{\Ltwo}
    + \frac{\eta^2}{\sqrt3}\big\|
    \bg_n\big\|_{\Ltwo}.\label{eq:42}
\end{align}
In view of \eqref{eq:thetaU}, we get
\begin{align}
    \big\|\bvartheta_{n+U} - \bvartheta_n
    \big\|_{\Ltwo}^2 &= \mathbb E\Big[
    \|(\eta/\gamma)(U \bar\alpha_1 \bv_n - U^2 \eta 
    \bar\beta_1\,\bg_n)\|^2
    + (2/3) \eta^3 U^3p/\gamma\Big]\\
    &\leqslant \frac{2\eta^2\|\bv_n\|_{\Ltwo}^2}{3\gamma^2} + 
    \frac{\eta^4\|\bg_n
    \|^2_{\Ltwo}}{10\gamma^2} + \frac{\eta^3 p}{6\gamma}.
\end{align}
The last claim of the lemma follows
from the previous display and \eqref{eq:42}.

\subsubsection{Proof of \Cref{lem:8}}
    From \eqref{eq:thetaU}, \eqref{eq:thetan+1}  and
    \eqref{eq:vn+1}, it follows that
    \begin{align}
        \gamma^2\|\bvartheta_{n+U}-\bvartheta_n\|_{\Ltwo}^2
        &\leqslant \|U\eta \bar\alpha_1 \bv_n -
        (U\eta)^2\bar\beta_1\,\bg_n\|^2_{\Ltwo} +(1/6)\eta^3 \gamma p\\
        & \leqslant  (2\eta^2/3)\|\bv_n\|^2_{\Ltwo} 
        + (\eta^4/10)\|\bg_n\|_{\Ltwo}^2 +(\eta^3/6)
        \gamma p\label{eq:44}
    \end{align}
    and 
    \begin{align}
        \gamma\|\bvartheta_{n+1}-\bvartheta_n\|_{\Ltwo}
        &\leqslant \eta\|\bv_n\|_{\Ltwo} + 0.5
        \eta^2\|\bg_{n+U}\|_{\Ltwo} +\sqrt{(2/3) \eta^3 \gamma p}\\
        &\leqslant \eta\|\bv_n\|_{\Ltwo} + 0.5
        \eta^2\|\bg_{n}\|_{\Ltwo} + 0.5M_\gamma\eta^2\gamma\|\bvartheta_{n+U}-\bvartheta_{n}\|_{\Ltwo}+\sqrt{(2/3)\eta^3 \gamma p}\\
        &\leqslant 1.001\eta\|\bv_n\|_{\Ltwo} + 0.501
        \eta^2\|\bg_{n}\|_{\Ltwo} +\sqrt{0.67\eta^3 \gamma p}\label{eq:45}
    \end{align}
    where in the last step we have used \eqref{eq:44} and the
    fact that $M_\gamma\eta^2/2\leqslant \eta ^2/8\leqslant 1/200$.
    A bit more precise computations also yield
    \begin{align}
        \gamma\|\bvartheta_{n+1}-\bvartheta_n\|_{\Ltwo}
        &\leqslant \big\{(\eta\|\bv_n\|_{\Ltwo} + 
         \eta^2\|\bar\beta_2\bg_n\|_{\Ltwo})^2 + 
         {\textstyle \frac23}\gamma\eta^3 p \big\}^{1/2} + 
         \frac{\eta^2}{2} \|\bg_{n+U} - \bg_n 
         \|_{\Ltwo} \\
         &\leqslant \Big\{\big(\eta\|\bv_n\|_{\Ltwo} + 
         {\textstyle\frac{\eta^2}{\sqrt{3}}} \|\bg_n
         \|_{\Ltwo} \big)^2 + {\textstyle \frac23}\gamma\eta^3 p \Big\}^{1/2} + \frac{M_\gamma\eta^2\gamma}{2} \|
         \bvartheta_{n+U} - \bvartheta_n\|_{\Ltwo}\\
         &\leqslant \big\{\big(\eta\|\bv_n\|_{\Ltwo} + 
         {\textstyle\frac{\eta^2}{\sqrt{3}}} \|\bg_n
         \|_{\Ltwo} \big)^2 + {\textstyle \frac23}\gamma\eta^3 p \big\}^{1/2} + \tfrac1{10}\eta^2\gamma \|
         \bvartheta_{n+U} - \bvartheta_n\|_{\Ltwo}.
    \end{align}
    Taking the squares of this inequality, we get
    \begin{align}
    \gamma^2\|\bvartheta_{n+1}-\bvartheta_n\|_{\Ltwo}^2
         &\leqslant \big( \eta\|\bv_n\|_{\Ltwo} + 
         {\textstyle\frac{\eta^2}{\sqrt{3}}} \|\bg_n
         \|_{\Ltwo} \big)^2 + {\textstyle \frac23}\gamma 
         \eta^3 p + \tfrac{1}{100} \eta^4\gamma^2\| \bvartheta_{n+U} -
         \bvartheta_n\|_{\Ltwo}^2\\ 
         & \quad + \tfrac15\eta^2\gamma\Big\{\big(\eta\|\bv_n\|_{\Ltwo} + 
         {\textstyle\frac{\eta^2}{\sqrt{3}}} \|\bg_n
         \|_{\Ltwo} \big)^2 + {\textstyle \frac23}\gamma\eta^3 p\Big\}^{1/2} \| \bvartheta_{n+U} - \bvartheta_n 
         \|_{\Ltwo}\\
         &\leqslant \big(1+ \tfrac1{10}\eta^2\big) 
         \eta^2\Big\{\big( \| \bv_n \|_{\Ltwo} + 
         {\textstyle\frac{\eta}{\sqrt{3}}} \|\bg_n
         \|_{\Ltwo} \big)^2 + {\textstyle \frac23}\eta \gamma p  + \tfrac1{10}\gamma^2 \| \bvartheta_{n+U} - \bvartheta_n\|_{\Ltwo}^2\Big\}\\
         &\leqslant 1.01\eta^2\Big\{\big( \|\bv_n\|_{\Ltwo} + 
         {\textstyle\frac{\eta}{\sqrt{3}}} \|\bg_n
         \|_{\Ltwo} \big)^2 + {\textstyle \frac23}\eta\gamma p+ 0.1\gamma^2 \|
         \bvartheta_{n+U} - \bvartheta_n\|_{\Ltwo}^2\Big\}\\
         &\leqslant 0.68\eta^2\big(3\|\bv_n\|^2_{\Ltwo} + \eta^2 
         \|\bg_n \|_{\Ltwo}^2 + \eta\gamma p\big). 
         \label{eq:48}
    \end{align}
    This implies that for $\gamma \geqslant 5M$ and
    $\eta\leqslant 1/5$, we have
    \begin{align}
        \|\bv_{n+1}\|_{\Ltwo}^2 & \leqslant (1-\bar\alpha_2\eta  
        )^2\|\bv_n\|_{\Ltwo}^2 - 4\eta(1-
        \bar\alpha_2\eta ) \mathbb E[\bar\beta_2\bv_n^\top\bg_{n+U}] + 
        \eta^2\|\bg_{n+U}\|_{\Ltwo}^2 + 2\eta\gamma  p\\
        &\qquad +  2\eta\sqrt{2\eta\gamma p}\|\bg_{n+U} -\bg_n\|_{\Ltwo}  \\
        & \leqslant (1-\bar\alpha_2\eta )^2\|\bv_n
        \|_{\Ltwo}^2 - 4\eta(1 - 
        \bar\alpha_2\eta ) \mathbb E[\bar\beta_2 
        \bv_n^\top
        \bg_{n+U}] + \eta^2\|\bg_{n+U} \|_{ 
        \Ltwo}^2 + 2\eta\gamma p\\
        &\qquad +  2M_\gamma \eta\sqrt{2\eta\gamma p}\|
        \gamma(\bvartheta_{n+U} -\bvartheta_n)\|_{\Ltwo}\\
        & \leqslant (1-\eta \bar\alpha_2)^2\|\bv_n
        \|_{\Ltwo}^2 - 4\eta(1-\eta \bar\alpha_2) 
        \mathbb E[\bar\beta_2\bv_n^\top\bg_{n+U}] + 
        \eta^2\|\bg_{n+U}\|_{\Ltwo}^2+ 2.1\eta\gamma p\\
        &\qquad + 20(M_\gamma \eta)^2\| 
        \gamma(\bvartheta_{n+U} -\bvartheta_n)\|_{\Ltwo}^2\\
        &\leqslant (1-\bar\alpha_2\eta )^2 \|\bv_n\|_{\Ltwo}^2 - 2\bar\alpha_2\eta  (1 -
        \eta \bar\alpha_2) \mathbb E [\bv_n^\top
        \bg_n] + 1.1\eta^2 \|\bg_n\|_{\Ltwo}^2 + 2.1\eta\gamma p\\
        &\qquad + 2M_\gamma \eta\|\bv_n\|_{\Ltwo} \|\gamma(\bvartheta_{n+U} -\bvartheta_n)\|_{\Ltwo} + 31 (M_\gamma \eta)^2 \|\gamma(\bvartheta_{n+U} -\bvartheta_n)\|_{\Ltwo}^2\\
        &\leqslant (1-\bar\alpha_2\eta )^2 \|\bv_n\|_{\Ltwo}^2 - 2\bar\alpha_2\eta  (1 -
        \eta \bar\alpha_2) \mathbb E [\bv_n^\top
        \bg_n] + 1.1\eta^2 \|\bg_n\|_{\Ltwo}^2 + 2.1\eta\gamma p\\
        &\qquad + 0.2 M_\gamma\eta^2\|\bv_n\|_{\Ltwo}^2 + 5M_\gamma
        \|\gamma(\bvartheta_{n+U} -\bvartheta_n)\|_{\Ltwo}^2 + 31(M_\gamma \eta)^2 \|\gamma(\bvartheta_{n+U} -\bvartheta_n)\|_{\Ltwo}^2\\
        &\leqslant (1-\bar\alpha_2\eta )^2 \|\bv_n\|_{\Ltwo}^2 - 2\bar\alpha_2\eta  (1 - \bar\alpha_2\eta ) \mathbb E [\bv_n^\top \bg_n] + 1.1 \eta^2 \|\bg_n\|_{\Ltwo}^2 + 2.1\eta\gamma p\\
        &\qquad + 0.2 M_\gamma\eta^2\|\bv_n\|_{\Ltwo}^2 + 5.4M_\gamma 
        \|\gamma(\bvartheta_{n+U} -\bvartheta_n)\|_{\Ltwo}^2.
    \end{align}
    Since for $\eta\leqslant 0.2$ we have
    $\bar\alpha_2\geqslant 0.9$, we get 
    $(1 - \bar\alpha_2\eta )^2 + 0.2M_\gamma \eta^2
    + 5.4M_\gamma(2\eta^2/3)\leqslant 
    (1-0.9\eta)^2 + 0.008\eta + 0.16\eta\leqslant 
    1- 1.47\eta$. Therefore,
    \begin{align}
         \|\bv_{n+1}\|_{\Ltwo}^2 & \leqslant {(1 - 1.47 \eta ) \| 
        \bv_n\|_{\Ltwo}^2 - 2\bar\alpha_2\eta  (1 - 
        \eta \bar\alpha_2) \mathbb E [\bv_n^\top \bg_{n}] + 1.12\eta^2 \|\bg_n\|_{\Ltwo}^2 +
        2.12\eta\gamma p}. 
        \label{eq:43}
    \end{align}
    The next step is to get an upper bound on 
    $\mathbb E[\bv_{n+1}^\top\bg_{n+1}] - \mathbb
    E[\bv_{n}^\top\bg_n]$ in order to prove
    \eqref{eq:help2}. To this end, we first
    note that
    \begin{align}
        \|\bv_{n+1}\|_{\Ltwo}
        &\leqslant \|\bv_n\|_{\Ltwo} + \eta\| 
        \bg_{n+U}\|_{\Ltwo}
        +\sqrt{2\eta\gamma p}\\
        &\leqslant \|\bv_n\|_{\Ltwo} + \eta\| 
        \bg_n\|_{\Ltwo} + M_\gamma \eta\| 
        \gamma(\bvartheta_{n+U} -\bvartheta_n)\|_{\Ltwo}
        +\sqrt{2\eta\gamma p}\\
        &\leqslant 1.004\big(\|\bv_n\|_{\Ltwo} +  
        \eta\|\bg_n\|_{\Ltwo} 
        +\sqrt{2\eta\gamma p}\big).\label{eq:46}
    \end{align}
    From \eqref{eq:45} and \eqref{eq:46}, we 
    also infer that
    \begin{align}
        \gamma\|\bv_{n+1}\|_{\Ltwo}\|\bvartheta_{n+1} - 
        \bvartheta_n\|_{\Ltwo} 
        &\leqslant  \eta\sqrt{3(1.001^2 + 0.501^2 + 0.34)}\big(\|\bv_n\|_{\Ltwo}^2+  
        \eta^2\|\bg_n\|_{\Ltwo}^2 + 
        2\eta\gamma  p\big)\\
        &\leqslant 2.2 \eta\big(\|\bv_n\|_{\Ltwo}^2+  
        \eta^2\|\bg_n\|_{\Ltwo}^2 + 
        2\eta\gamma  p\big).
    \end{align}
    Therefore, this bound and some elementary 
    computations yield
    \begin{align}
        \mathbb E[\bv_{n+1}^\top\bg_{n+1}] &\leqslant 
        \mathbb E[\bv_{n}^\top\bg_n]  + \mathbb E[\bv_{n+1}^\top(\bg_{n+1}-\bg_n)] + \mathbb E[(\bv_{n+1} -\bv_n)^\top\bg_n]\\
        &\leqslant \mathbb E[\bv_{n}^\top\bg_n] + M_\gamma\|\bv_{n+1}\|_{\Ltwo}\|\gamma(\bvartheta_{n+1}-\bvartheta_n)\|_{\Ltwo} -\bar\alpha_2\eta \mathbb E[\bv_n^\top\bg_n] -\eta\mathbb E [\bg_n^\top\bg_{n+U}]\\
        &\leqslant (1- \bar\alpha_2\eta )\mathbb E[\bv_{n}^\top\bg_n] + M_\gamma\|\bv_{n+1}\|_{\Ltwo}\|\gamma(\bvartheta_{n+1}-\bvartheta_n)\|_{\Ltwo} - \eta\|\bg_n\|_{\Ltwo}^2\\ 
        &\qquad + M_\gamma \eta\|\bg_n\|_{\Ltwo}\|\gamma(\bvartheta_{n+U} - \bvartheta_n)\|_{\Ltwo}\\
        &\leqslant (1- \bar\alpha_2\eta )\mathbb E[\bv_{n}^\top\bg_n] - \eta\|\bg_n\|_{\Ltwo}^2 + 2.2 M_\gamma \eta\big(\|\bv_n\|_{\Ltwo}^2 +  
        \eta^2\|\bg_n\|_{\Ltwo}^2 + 
        2\eta\gamma  p\big)\\
        &\qquad + M_\gamma \eta^2\|\bg_n\|_{\Ltwo} 
        \big(\tfrac23\,\|\bv_n\|^2_{\Ltwo} 
        + \tfrac1{10}\eta^2\|\bg_n\|_{\Ltwo}^2 +
        \tfrac16 \eta\gamma  p\big)^{1/2}\\
        &\leqslant (1- \bar\alpha_2\eta )\mathbb E[\bv_{n}^\top\bg_n] - \eta\|\bg_n\|_{\Ltwo}^2 + 2.2 M_\gamma
        \eta\big(\|\bv_n\|_{\Ltwo}^2 +  
        \eta^2\|\bg_n\|_{\Ltwo}^2 + 
        2\eta\gamma  p\big)\\
        &\qquad+ 0.5 M_\gamma \eta \big(\tfrac23\,\|\bv_n\|^2_{\Ltwo} 
        + \tfrac{11}{10}\eta^2\|\bg_n\|_{\Ltwo}^2 +
        \tfrac16 \eta\gamma  p\big).
    \end{align}
    Grouping the terms, and using the fact that 
    $M_\gamma\eta\leqslant 1/25$, we arrive at
    \begin{align}
        \mathbb E[\bv_{n+1}^\top\bg_{n+1}] 
        &\leqslant {(1- \bar\alpha_2\eta )\mathbb E[\bv_{n}^\top\bg_n] - 0.97\eta
        \|\bg_n\|_{\Ltwo}^2  + 2.54 M_\gamma \eta\|\bv_n\|_{\Ltwo
        }^2 +  4.5 M_\gamma \gamma \eta^2 p}\\
        &\leqslant {(1- \bar\alpha_2\eta )\mathbb E[\bv_{n}^\top\bg_n] - 0.97\eta
        \|\bg_n\|_{\Ltwo}^2  + 0.51  \eta\|\bv_n\|_{\Ltwo
        }^2 +  0.9 \gamma \eta^2 p}.\label{eq:47}
    \end{align}
    Similarly, using the Lipschitz property of $\nabla f$
    and \eqref{eq:48}, we get
    \begin{align}
        \gamma\mathbb E[f_{n+1} - f_n] & \leqslant 
        \gamma\mathbb E[ \bg_n^\top(\bvartheta_{n+1} 
        - \bvartheta_n)] + (M_\gamma/2)\| \gamma(\bvartheta_{n+1} -
        \bvartheta_n) \|^2_{\Ltwo}\\
        & =   \mathbb E[\bg_n^\top 
        (\bar\alpha_2\eta \bv_n -\eta^2\bar\beta_2 \bg_{n+U})] + 0.07\eta^2( 3\|\bv_n\|^2_{\Ltwo} + \eta^2 \|\bg_n \|_{\Ltwo}^2 +  \eta\gamma  p) \\
        &\leqslant \bar\alpha_2\eta \mathbb
        E[\bv_n^\top \bg_n] -  \eta^2\mathbb E[\bar\beta_2]
        \|\bg_n\|^2_{\Ltwo} 
        +  0.2\eta^2\|\bg_n\|_{\Ltwo}
        \|\gamma(\bvartheta_{n+U}- \bvartheta_n)\|_{\Ltwo}\\
        &\qquad + 0.07\eta^2( 3\|\bv_n\|^2_{\Ltwo} + \eta^2 \|\bg_n \|_{\Ltwo}^2 +  \eta\gamma  p) \\
        &\leqslant \bar\alpha_2\eta \mathbb
        E[\bv_n^\top \bg_n] -  \eta^2\mathbb E[\bar\beta_2]
        \|\bg_n\|^2_{\Ltwo} + 0.1\eta^4\|\bg_n
        \|_{\Ltwo}^2 + 0.1 \|\gamma(\bvartheta_{n+U}- \bvartheta_n) 
        \|_{\Ltwo}^2\\
        &\qquad + 0.07\eta^2( 3\|\bv_n\|^2_{\Ltwo} + \eta^2 \|\bg_n \|_{\Ltwo}^2 +  \eta\gamma  p)\\
        &\leqslant \bar\alpha_2\eta \mathbb
        E[\bv_n^\top \bg_n] - 0.468  \eta^2
        \|\bg_n\|^2_{\Ltwo} + 0.1 \eta^4\|\bg_n
        \|_{\Ltwo}^2\\
        &\qquad + 0.07 \eta^2 \|\bv_n\|^2_{\Ltwo} + 
        0.01 \eta^4 \|\bg_n \|_{\Ltwo}^2 + 0.02 \eta^3\gamma  p\\
        &\qquad + 0.07\eta^2( 3\|\bv_n\|^2_{\Ltwo} + \eta^2 \|\bg_n \|_{\Ltwo}^2 +  \eta\gamma  p).
    \end{align}
    Grouping the terms, and using the fact that 
    $M_\gamma\eta^2\leqslant 1/50$, we arrive at
    \begin{align}
        \gamma \mathbb E[f_{n+1} - f_n] 
        &\leqslant \bar\alpha_2\eta \mathbb
        E[\bv_n^\top \bg_n] - 0.46\eta^2\|\bg_n\|^2_{\Ltwo} +  0.28\eta^2 \|\bv_n\|^2_{\Ltwo} + 
         0.09 \eta^3\gamma  p.
    \end{align}
    This completes the proof of the lemma. 